\documentclass[a4paper,10pt]{amsart}

\usepackage{a4wide}
\usepackage[latin1]{inputenc} 
\usepackage[T1]{fontenc}

\usepackage{amsfonts}
\usepackage{amsmath,amssymb,amsthm,stmaryrd}

\usepackage{mathrsfs}
\usepackage{esint}

\usepackage{fancyhdr}
\usepackage{verbatim}

\usepackage{graphicx}
\usepackage{ifpdf}
\ifpdf
\fi 

\DeclareMathOperator{\diam}{diam}

\newcommand{\essinf}[0]{\operatornamewithlimits{ess\,inf}}

\providecommand{\abs}[1]{\lvert#1\rvert} 

\newcommand{\loc}[0]{\operatorname{loc}}

\newcommand{\R}{\mathbb{R}}

\newcommand{\Z}{\mathbb{Z}}

\swapnumbers
\numberwithin{equation}{section}

\makeatletter
  \let\c@equation\c@subsection
\makeatother

\theoremstyle{plain}
\newtheorem{theorem}[subsection]{Theorem}
\newtheorem{lemma}[subsection]{Lemma}

\newtheorem{proposition}[subsection]{Proposition}

\theoremstyle{definition}
\newtheorem{definition}[subsection]{Definition}

\theoremstyle{remark}
\newtheorem{remark}[subsection]{Remark}

\makeatletter
\@namedef{subjclassname@2010}{%
  \textup{2010} Mathematics Subject Classification}
\makeatother

\title[Two-weight norm inequalities]{Two-weight norm inequalities for Potential type and maximal operators in a metric space}
\author{Anna Kairema}
\address{Department of Mathematics and Statistics, P.O.B. 68 (Gustaf H\"allstr\"omin katu 2), FI-00014 University of Helsinki, Finland}
\email{anna.kairema@helsinki.fi}

\thanks{The author is supported by the Academy of Finland, grant 133264.}

\begin{document}

\begin{abstract}
We characterize two-weight norm inequalities for potential type integral operators in terms of Sawyer-type testing conditions. Our result is stated in a space of homogeneous type with no additional geometric assumptions, such as group structure or non-empty annulus property, which appeared in earlier works on the subject. One of the new ingredients in the proof is the use of a finite collection of adjacent dyadic systems recently constructed by the author and T. Hyt\"onen. We further extend the previous Euclidean characterization of two-weight norm inequalities for fractional maximal functions into spaces of homogeneous type.
\end{abstract}

\maketitle

\section{Introduction}
Dyadic Harmonic Analysis has received a renewed attention in recent years, spurred by S. Petermichl's study \cite{P2000} on Haar shifts which can be used to prove deep results about the Hilbert transform and other classical operators in the Euclidean space. The developments in this area culminated to T. Hyt\"onen's Dyadic Representation Theorem \cite{TH2012}, which provides a direct link between Classical and Dyadic Analysis by showing that any Calderón--Zygmund singular integral operator has a representation in terms of certain simpler  dyadic shift operators. This gives a new insight into the fine structure of such operators and provides a tool to prove some substantial new results, among them the $A_2$ conjecture which so far was a key problem in the weighted theory. 

This dyadic approach has, in particular, been exploited in the study of $L^p$ boundedness of positive operators. The key step is the approximation of the operator by simpler dyadic model operators. Some cleverly constructed model operators are yet rich enough so that the original theorems can be recovered from their dyadic analogues. Hence, dyadic cubes pose a substantial tool in Euclidean Analysis for discretizing objects and thereby reducing problems into a parallel dyadic world where objects, statements and analysis are often easier. 

Constructions of dyadic cubes in metric spaces, led by M. Christ \cite{Christ90} and continued in \cite{oma, HM09}, have made this approach available in more general settings allowing some easy extensions of Euclidean results into more general metric spaces. Dyadic theorems have the virtue of remaining true in a very general framework; an Euclidean dyadic argument may often, with virtually no extra effort, be carried over into more general metric spaces. The dyadic structure, in particular the simple inclusion properties of dyadic cubes, then play the main role in the argumentation. 

However, the passage from dyadic model operators into the original one has usually entailed some extra structure on the space in addition to the standard setting of a space of homogeneous type. In particular, in the previous works by E. T. Sawyer, I. E. Verbitsky, R. L. Wheeden and S. Zhao \cite{SW92, SWZ96, VW98} on norm estimates for potential type operators, the space was assumed to have a certain group structure so as to allow the translations of the dyadic lattice. In fact, the recovery of the classical-style operator from its dyadic counterparts seems to require not just one dyadic system but several adjacent systems. In the present paper, the recovery of potential type operators from suitably defined dyadic model operators is obtained by some recent results on such adjacent families of dyadic cubes. As an application, we derive characterizations of two-weight norm inequalities by means of Sawyer-type testing condition. 

\subsection{Set-up: spaces and operators}\label{setup1}
Let $(X,\rho)$ denote a quasi-metric space and let $\sigma$ and $\omega$ be positive Borel-measures on $X$. We assume that all balls are measurable with finite measure. This implies that our measures are $\sigma$-finite and that the set of atoms (point masses; points $x\in X$ with $\sigma(\{x\})>0$) is at most countable. No additional assumptions are imposed on measures unless otherwise indicated. In examples and in Section~\ref{sec:MO} we will consider measures $\mu$ which satisfy the \textit{doubling condition} that
\begin{equation}\label{def:doubling}
0<\mu(B(x,2r))\leq C_\mu \mu (B(x,r))<\infty\quad\text{for all } x\in X,r>0,
\end{equation}
with a constant $C_\mu>0$ that is independent of $x$ and $r$. A quasi-metric space $(X,\rho)$ with a doubling measure $\mu$ is called \textit{a space of homogeneous type}.

Let $1<p\leq q<\infty$. We study integral operators $S$ acting on suitable functions on $X$, and derive a characterization of the two-weight strong type norm inequality
\begin{equation}\label{eq:strongtype_S(1)}
\left(\int_{X} (S(f\,d\sigma))^q\,d\omega\right)^{1/q}\leq C\left( \int_{X}f^p\,d\sigma\right)^{1/p},\quad f\in L^p_\sigma .
\end{equation}

Our characterizations are in terms of ``testing type'' conditions, first introduced by E. T. Sawyer \cite{S82} in relation to the Hardy--Littlewood maximal function, which involve certain obviously necessary conditions; in order to have the full norm inequality \eqref{eq:strongtype_S(1)}, it suffices to have such an inequality for special test functions only:

\begin{definition}[$(E,F,G)$ testing condition]
We say that operator $S$ satisfies an $(E,F,G)$ \textit{testing condition} with arbitrary sets $E,F$ and $G$ from some collections of measurable sets in $X$ if
\begin{equation*}
\left(\int_{E} (S(\chi_{F}\,d\sigma))^q\,d\omega\right)^{1/q}\leq C\sigma(G)^{1/p}
\end{equation*}
holds for all such $E,F$ and $G$ with a constant $C$ which is independent of the sets. 
\end{definition}
Typical examples include $(E,F,G)\in \{(X,B,B), (B,B,B), (X,Q,Q), (Q,Q,Q)\}$ where $B$ denotes an arbitrary ball and $Q$ an arbitrary (dyadic) cube.

As an important special case, let us consider measures $\sigma$ and $\omega$ which are both absolutely continuous with respect to an underlying measure $\mu$. Then the inequality  \eqref{eq:strongtype_S(1)} reduces to the two-weight norm inequality
\begin{equation}\label{eq:strongtype_S(2)}
\left(\int_{X}(S(f\,d\mu))^q\, wd\mu\right)^{1/q}\leq 
C\left(\int_{X}f^p\, ud\mu\right)^{1/p}
\end{equation}
by choosing $d\omega = w d\mu$ and $d\sigma = u^{-1/(p-1)}d\mu$, $u=(d\sigma/d\mu)^{1-p}$, and replacing $f$ by $fu^{1/(p-1)}$.  

The characterization of norm estimates \eqref{eq:strongtype_S(1)} and \eqref{eq:strongtype_S(2)} by means of testing conditions has been studied in depth for many classical operators in both the Euclidean space and more general metric spaces. For many operators these characterizations involve the adjoint operator $S^\ast$ which is defined under the usual pairing, i.e.
\[\int_{X}(S(f\,d\sigma))g\,d\omega = \int_{X}f(S^\ast (g\,d\omega))\,d\sigma \quad \text{for all $f$ and $g$.}\]
We say that $S$ satisfies a \textit{dual $(E,F,G)$ testing condition} if
\begin{equation*}
\left(\int_{E} (S^\ast(\chi_{F}\,d\omega))^{p'}\,d\sigma\right)^{1/p'}\leq C\omega(G)^{1/q'},
\end{equation*}
where $1/p+1/p'=1$ and $1/q+1/q'=1$.


\bigskip

Our main results concern a large class of positive operators of the following type:
\begin{definition}[Potential type operator]\label{def:PO}
We say that operator $T$ is \textit{an operator of potential type} if it is of the form
\begin{equation}\label{def:operatorT}
T(f\, d\sigma)(x)=\int_{X}K(x,y)f(y)\, d\sigma(y), \quad x\in X,
\end{equation}
where the kernel $K\colon X\times X\to [0,\infty]$ is a non-negative function which satisfies the following monotonicity conditions: For every $k_2>1$ there exists $k_1>1$ such that 
\begin{equation}\label{kernel}
\begin{split}
&K(x,y)\leq k_1 K(x',y)\quad \text{whenever $\rho (x',y)\leq k_2\rho (x,y)$},\\
&K(x,y)\leq k_1 K(x,y')\quad \text{whenever $\rho (x,y')\leq k_2\rho (x,y)$}.
\end{split}
\end{equation} 
\end{definition}
We shall denote the formal adjoint of $T$ by $T^\ast$, which is given by
\begin{equation*}
T^\ast (g\, d\omega)(y)=\int_{X}K(x,y)g(x)\,d\omega(x), \quad y\in X.
\end{equation*}

\subsubsection{Examples of operators}
Important examples of potential type operators are provided by \textit{fractional integrals} which over quasi-metric measure spaces $(X,\rho,\mu)$ are known to be considered in different forms. One common and widely studied notation; see e.g. the book \cite{EKM} and the paper \cite{G06}, is given by the formula
\[T^n_\alpha f(x):= \int_{X}\frac{f(y)\,d\mu(y)}{\rho(x,y)^{n-\alpha}},\quad 0<\alpha<n,\]
and it has been studied in both the doubling \cite{GattoVagi1, GSV96, KK89} and non-doubling \cite{GCG03, KM01, KM05} case. 
Here the parameter $n>0$ is related to the ``dimension'' of $\mu$ through the growth condition
\[\mu(B(x,r))\leq Cr^n,\quad x\in X,r>0.\] 
Another type of fractional integral, which fits into the present context, is given by
\[T_\gamma f(x):= \int_{X}\frac{f(y)\,d\mu(y)}{\mu(B(x,\rho(x,y)))^{1-\gamma}},\quad 0<\gamma<1 .\]
This operator is considered e.g. in the book \cite{EKM} and the papers \cite{BC94} and \cite[Section 4.1]{KS09}, and most recently in \cite{Kairema12}. In particular, for $X=\Z$ with the counting measure,
\[T_\gamma f(x)=\sum_{y\in \Z}\frac{f(y)}{(1+\abs{x-y})^{1-\gamma}}. \]
A kind of hybrid of the two operators $T^n_\alpha$ and $T_\gamma$,
\[\mathfrak{T}^\alpha f(x):= \int_{X}\frac{\rho(x,y)^\alpha}{\mu(B(x,\rho(x,y)))}f(y) \,d\mu(y), \quad \alpha>0, \]
is studied e.g. in the book \cite{GGKK98} and the paper \cite{AS08}. The operator $\mathfrak{T}^\alpha$ does not, in general, get into the present context. However, if $\mu$ satisfies the doubling condition \eqref{def:doubling} and, in addition, the reverse doubling type condition that
\[\mu(B(x,kr))\geq Ck^\alpha \mu(B(x,r))\quad\text{for all } x\in X, \; r,k>0,\]
then $\mathfrak{T}^\alpha$ is a potential type operator defined in \ref{def:PO}. 
Also note that if $\mu$ satisfies the well-established regularity condition that
\[cr^n\leq \mu(B(x,r))\leq Cr^n, \quad \text{for all }x\in X, r>0 \text{ and for some } c,C,n>0,\]
then all the three operators mentioned are equivalent. In particular, in the usual Euclidean space $\R^n$ with the Lebesgue measure, all the three operators reduce to the usual fractional integrals or Riesz potentials,
\[I_\alpha f(x):= \int_{\R^n}\frac{f(y)}{\abs{x-y}^{n-\alpha}}dy\quad 0<\alpha<n , \]
which are the basic examples of potential type operators. 

For other examples of operators defined in \ref{def:PO}; see \cite{PerezWheeden03} and the references listed in \cite[pp. 819--820]{SW92}. 

\bigskip

Weighted norm inequalities for $I_\alpha$ have been treated by several authors. The characterizations of the general two-weight weak and strong type estimate in the case $1<p\leq q<\infty$ and $X=\R^n$ are due to Sawyer \cite{S84, S88}. Analogous characterizations for more general (quasi-)metric spaces and for more general potential type operators can be found in \cite{VW98, WZ96} for weak type estimates, and in \cite{SW92, SWZ96, VW98, WZ96} for strong type estimates.


\subsection{Earlier results in metric spaces}
In the previous papers mentioned above, the framework for the study of potential type operators is as follows:
\begin{definition}[A Sawyer--Wheeden type space]\label{def:SWspace}
Let $(X,\rho,\mu)$ be a space of homogeneous type. Suppose that the space has the following additional properties:
\begin{enumerate}
\item[(1)] $X$ has the geometric property that all the annuli $B(x,R)\setminus B(x,r)$ are non-empty for $0<r<R$ and $x\in X$ (we call this  \textit{the non-empty annuli property});
\item[(2)] the measures $\sigma$ and $\omega$ appearing in the two-weight norm inequality \eqref{eq:strongtype_S(1)} vanish on sets which consist of an individual point (the measures do not have atoms). 
\end{enumerate}
As to make the comparison between our results and the earlier related results more distinct, we shall refer to such spaces $(X,\rho,\mu;\sigma,\omega)$ as \textit{Sawyer--Wheeden type spaces} according to the authors of the paper \cite{SW92} which is one of our early references on the topic. 
\end{definition}
Let us record some of what is known about the characterization of \eqref{eq:strongtype_S(1)} for $T$ by means of testing conditions in Sawyer--Wheeden type spaces. 

Wheeden and Zhao \cite[Theorem 1.4]{WZ96} characterized \eqref{eq:strongtype_S(1)} with $S=T$ by a $(B,B,B)$ testing condition together with a dual $(B,B,B)$ testing condition for balls $B$. There has been interest in finding an analogous characterization by testing conditions involving ``cubes'' instead of balls. First, Sawyer, Wheeden and Zhao \cite[Theorem 1.1]{SWZ96} showed that \eqref{eq:strongtype_S(1)} with $S=T$ is characterized by a $(X,Q,Q)$ testing condition together with a dual $(X,Q,Q)$ testing condition, improving some earlier results of Sawyer \cite{S88} and Sawyer and Wheeden \cite{SW92}. On the other hand, it is \textit{not} sufficient to replace the integration over $X$ in either of the testing inequalities by integration over $Q$ (for dyadic $Q$), even in the Euclidean case; a counterexample was given in \cite[Example 1.9]{SWZ96}. The authors, however, provided some results involving testing conditions with dyadic cubes which are weaker: Under the additional technical assumption that $T(\chi_B\, d\sigma)\in L^q_\omega$ for all balls $B$, \eqref{eq:strongtype_S(1)} with $S=T$ is characterized by a $(E,Q,Q)$ testing condition together with a dual $(E,Q,F)$ testing condition where $E$ and $F$ are appropriate enlargements of $Q$ (for an arbitrary dyadic $Q$); for a specific result of this kind, see \cite[Theorem 1.2]{SWZ96}.

Most precise results are obtained by reducing to appropriate dyadic model operators. While the reduction of \eqref{eq:strongtype_S(1)} to testing conditions is admissible in a very general setting for these dyadic operators, the recovery of the ``classical-style'' operator and thus, the return to the original norm estimate, has in the previous papers required stronger assumptions on the space, as mentioned. In particular, Verbitsky and Wheeden made the additional assumption that $X$ has an appropriate \textit{group structure} with respect to a group operation $``+"$ (see \cite[Theorem 4]{SW92} for precise definitions), and obtained a $(Q,Q,Q)$ characterization with integration over all the translates of dyadic cubes: The full norm inequality \eqref{eq:strongtype_S(1)} with $S=T$ holds, if and only if both
\begin{equation*}
\left( \int_{Q+z}T(\chi_{Q+z}\, d\sigma)^q\, d\omega \right)^{1/q}\leq C\sigma (Q+z)^{1/p},
\end{equation*}
and
\begin{equation*}
\left( \int_{Q+z}T^\ast(\chi_{Q+z}\, d\omega)^{p'}\, d\sigma \right)^{1/p'}\leq C\omega (Q+z)^{1/q'},
\end{equation*}
hold for all dyadic cubes $Q$ and all $z\in X$; see \cite[Theorem 1.2]{VW98} which improves the earlier related result \cite[Theorem 1.3]{SWZ96}.

\subsection{Aims of the present paper}
We continue on the investigations of Sawyer, Wheeden and Zhao \cite{SWZ96} and Verbitsky and Wheeden \cite{VW98}. The present contribution consists of weakening of the hypotheses as follows: First, our result does not require an underlying doubling measure, only a weaker geometric doubling property (precise definition will be given in Section \ref{subsec:setup}). Second, we do not assume any group structure on $X$. We will further drop the geometric non-empty annuli property assumption as well as consider more general measures by allowing atoms. 

\subsubsection{Examples of spaces}
(1) Suppose that $(X,\rho,\mu)$ is a space of homogeneous type. Then, if $X$ is bounded or has atoms (or isolated points), there are always some empty annuli.

(2) $(\Z,\abs{\cdot})$ does not have the non-empty annuli property. If $\mu$ is a counting measure, then every point in $(\Z, \mu)$ is an atom.

(3) Interesting examples of spaces of homogeneous type which have no group structure arise when we consider domains $\Omega$ in a space of homogeneous type $(X,\rho,\mu)$ which have the following ``plumpness'' property: For all $x\in \Omega$ and $r\in (0,\diam \Omega)$, there exists $z\in X$ with $B(z,cr)\subseteq B(x,r)\cap\Omega$ where $c\in (0,1)$ is independent of $x$ and $r$. Then $(\Omega,\mu\vert_{\Omega})$ is a space of homogeneous type. Indeed, if $x\in \Omega, r>0$ and $z\in\Omega$ is a point provided by plumpness,
\[\mu\vert_{\Omega}(B(x,r))\geq \mu\vert_{\Omega}(B(z,cr))=\mu(B(z,cr))\geq C\mu(B(z,3A_0r))\]
with $C=C(A_0,c,\mu)$ since $B(z,cr)\subseteq B(x,r)\cap\Omega$ and $\mu$ is doubling. We note that $B(x,2r)\subseteq B(z,3A_0r)$, which yields
\[\mu\vert_{\Omega}(B(x,r))\geq C\mu(B(x,2r))\geq C\mu\vert_{\Omega}(B(x,2r)). \] Even if $X$ has group structure, this is easily lost in a subset.

\bigskip

Even though testing with balls seems especially natural in the metric space context, there has been interest in finding characterizations for the norm inequality \eqref{eq:strongtype_S(1)} with testing conditions involving dyadic cubes, as mentioned. In particular, these characterizations have had useful applications to half-space estimates; see the comments following \cite[Theorem 1.6]{SWZ96}. 

As the counterexample \cite[Example 1.9]{SWZ96} shows, testing conditions with just one family of dyadic cubes is not enough to obtain the full norm inequality \eqref{eq:strongtype_S(1)}. Thus, a larger collection of cubes is required. One of the new ingredients in our approach is the use of a finite collection $\{\mathscr{D}^t\}_{t=1}^{L}$ of adjacent systems $\mathscr{D}^t$ of dyadic cubes with the following properties \cite{oma}: individually, each family has the features of the dyadic ``cubes'' introduced by M. Christ; collectively, any ball $B$ is contained in some dyadic cube $Q\in \mathscr{D}^t$ in one of the systems with side length at most a fixed multiple times the radius of $B$; a more precise description of the adjacent systems will be given in Section~\ref{sec:adjacent}. These will allow us to only test over a ``representative'' collection of countably many cubes instead of all the translates of the dyadic lattice, which appeared in the previous papers on the topic. Our main result is the following:
\begin{theorem}\label{thm:theoremB}
Let $1<p\leq q <\infty$, and let $\sigma$ and $\omega$ be positive Borel-measures on $(X,\rho)$ with the property that $\sigma(B)<\infty$ and $\omega(B)<\infty$ for all balls $B$. Let $T$ be a potential type operator. Then
\begin{equation*}
\| T\|_{L^p_\sigma\to L^q_\omega}\approx [\sigma ,\omega]_{S_{p,q}}+ [\omega ,\sigma]^{\ast}_{S_{q',p'}},
\end{equation*}
and the constants of equivalence only depend on the geometric structure of $X$, and $p$ and $q$. Here
\[[\sigma ,\omega]_{S_{p,q}}:= \sup_{Q}\sigma(Q)^{-1/p}
\left\| \chi_{Q}T(\chi_{Q}\, d\sigma)\right\|_{L^q_\omega} 
\]
and
\[[\omega ,\sigma]^{\ast}_{S_{q',p'}}:= \sup_{Q}\omega(Q)^{-1/q'}
\left\| \chi_{Q}T^\ast(\chi_{Q}\, d\omega)\right\|_{L^{p'}_\sigma} \]
are the testing conditions where the supremum is over all dyadic cubes $Q\in \bigcup_{t=1}^{L}\mathscr{D}^t$, and $\infty\cdot 0$ is interpreted as $0$. 
\end{theorem}
 We will construct dyadic model operators associated to $T$ and each dyadic system $\mathscr{D}^t$. In turn, the original operator is pointwise equivalent to a sum of these discrete models over the collection of adjacent dyadic systems. Having this, the two-weight norm inequalities for $T$ are governed by the ones for the dyadic models. From here, the existing techniques can be further pushed to yield the desired estimates. We also characterize the corresponding weak type norm inequalities for potential type operators. We emphasize that the fact that our measures are allowed with atoms entail some extra considerations in the proofs, whereas the main results apply to any measure space, as described, whether atom free or with atoms, or even to spaces consisting only of atoms, such as $\Z$. Applications of these characterizations will be considered in a forthcoming paper by the author. 

We further provide similar characterizations of norm inequalities for the fractional maximal operators extending the Euclidean characterization due to Sawyer \cite{S82} into more general metric spaces. 

\bigskip 

\textit{Acknowledgements:} This paper has been supported by the Academy of Finland, project 133264. The paper is part of the author's PhD thesis written under the supervision of Associate professor Tuomas Hyt\"onen. The author wishes to express her greatest gratitude for the anonymous referee for their helpful comments and remarks.


\section{Definitions, notations and geometric lemmas}\label{def;not;geom}
\subsection{Set-up}\label{subsec:setup} Let $\rho$ be a quasi-metric on the space $X$, i.e. it satisfies the axioms of a metric except for the triangle inequality, which is assumed in the weaker form
\begin{equation*}
  \rho(x,y)\leq A_0\big(\rho(x,z)+\rho(z,y)\big), \quad x,y,z\in X,
\end{equation*}
with a constant $A_0\geq 1$ independent of the points.
The quasi-metric space $(X,\rho)$ is assumed to have the following \emph{geometric doubling property}: There exists a positive integer $A_1$ such that for every $x\in X$ and $r>0$, the ball $B(x,r):=\{y\in X:\rho(y,x)<r\}$ can be covered by at most $A_1$ balls $B(x_i,r/2)$; we recall the well-known result that measure doubling implies geometrical doubling so that Sawyer--Wheeden type spaces also enjoy this (weaker) geometric doubling property. As usual, if $B=B(x,r)$ and $c>0$, we denote by $cB$ the ball $B(x,cr)$. 
The assumptions on measures are as in \ref{setup1}. 

\subsection{The adjacent dyadic systems.}\label{sec:adjacent} Continuing earlier work of M. Christ \cite{Christ90} and T. Hyt\"onen and H. Martikainen \cite{HM09}, it was shown in  \cite{oma} that a geometrically doubling space $(X,\rho)$ has \textit{a dyadic structure}: Given a fixed parameter $\delta\in (0,1)$ which satisfies $96A_0^6\delta\leq 1$ and a fixed point $x_0\in X$, we may construct a finite collection of families $\mathscr{D}^t, t=1,\ldots ,L=L(A_0,A_1,\delta)<\infty$, called \textit{the adjacent dyadic systems}. Individually, each system $\mathscr{D}^t$ has the features of the dyadic ``cubes'' introduced by Christ: $\mathscr{D}^t$ is a countable family of Borel sets $Q^k_\alpha, k\in \Z, \alpha\in \mathscr{A}_k$, called \textit{dyadic cubes}, which are associated with points $z^k_\alpha$, and have the properties that
\begin{equation}\label{eq:cover}
  X=\bigcup_{\alpha}Q^k_{\alpha}\quad\text{(disjoint union)}\quad\forall k\in\Z;
\end{equation}
\begin{equation}\label{eq:nested}
   \text{if }\ell\geq k\text{, then either }Q^{\ell}_{\beta}\subseteq Q^k_{\alpha}\text{ or }Q^k_{\alpha}\cap Q^{\ell}_{\beta}=\varnothing;
\end{equation}
\begin{equation}\label{eq:contain}
  B(z^k_{\alpha},c_1\delta^k)\subseteq Q^k_{\alpha}\subseteq B(z^k_{\alpha},C_1\delta^k)=:B(Q^k_{\alpha})\text{, where $c_1:=(12A_0^4)^{-1}$ and $C_1:=4A_0^2$};
\end{equation}
\begin{equation}\label{eq:monotone}
   \text{if }\ell\geq k\text{ and }Q^{\ell}_{\beta}\subseteq Q^k_{\alpha}\text{, then }B(Q^{\ell}_{\beta})\subseteq B(Q^k_{\alpha});
\end{equation}
\begin{align}\label{eq:fixed_point}
\forall k\in \Z, \text{ there exists $\alpha$ such that }
x_0=z^k_\alpha, \text{ the center point of } Q^k_\alpha . 
\end{align}
Collectively, the collection $\{\mathscr{D}^t \}_{t=1}^{L}$ has the following property:
\begin{align}\label{property:ball;included}
\text{For every ball }B(x,r)\subseteq X &\text{ with }\delta^{k+2}<r\leq \delta^{k+1},\text{ there exists $t$ and $Q^k_\alpha\in \mathscr{D}^t$ such that }\nonumber \\
&B(x,r)\subseteq Q^k_\alpha \text{ and } \diam Q^k_\alpha \leq Cr.
\end{align}
Constant $C\geq 1$ in \eqref{property:ball;included} depends only on $A_0$ and $\delta$ (and we may choose $C=8A_0^3\delta^{-2}$). 

We say that the set $Q^k_\alpha$ is a dyadic cube of generation $k$ centred at $z^k_\alpha$. Given $t$ and $x\in X$, we denote by $Q^k(x,t)$ the (unique) dyadic cube of generation $k$ in $\mathscr{D}^t$ that contains $x$. 

It is important to notice that a dyadic cube $Q^k_\alpha$ is identified by the index pair $(k,\alpha)$ rather than as a set of points. Accordingly, there might occur repetition in the collection $\mathscr{D}^t$ in the sense that for two cubes $Q^k_\alpha, Q^\ell_\beta \in \mathscr{D}^t$, we might have that $(k,\alpha)\neq (\ell,\beta)$ but $Q^k_\alpha=Q^\ell_\beta$. This aspect is to be taken into consideration in the proof of our main result; cf. Lemmata \ref{lem:mainlemma} and \ref{lem:sparseness}.

\begin{remark}
We mention that, by carefully reading the proof of \cite[Theorem 4.1]{oma}, one may acquire an upper bound for $L$ (the number of the adjacent families) which depends on the parameters $A_0$ (the quasi-metric constant), $A_1$ (the geometric doubling constant) and $\delta$. In fact,
\begin{equation}\label{eq:upperbound}
L=L(A_0,A_1,\delta)\leq A_1^6(A_0^4/\delta)^{\log_2A_1}.
\end{equation}
There is, however, no reason to believe that \eqref{eq:upperbound} is, by any means, optimal. In the Euclidean space $\R^n$ with the usual structure we have $A_0=1$ and $A_1\geq 2^n$, and $\delta=\frac{1}{2}$, so that \eqref{eq:upperbound} yields an upper bound of order $2^{7n}$. However, T. Mei \cite{tMei} has shown that the conclusion \eqref{property:ball;included} can be obtained with just $n+1$ cleverly chosen systems $\mathscr{D}^t$. As for now, no better bound than \eqref{eq:upperbound} is known for general metric spaces.
\end{remark}

From now on the point $x_0\in X$ and the parameter $\delta>0$ will be fixed, and $\delta$ is assumed to satisfy $96A_0^6\delta\leq 1$. We will consider a fixed collection $\{\mathscr{D}^t\}$ provided by \cite{oma}, where each $\mathscr{D}^t$ satisfies the properties listed in \eqref{eq:cover}--\eqref{eq:fixed_point}, and the collection $\{\mathscr{D}^t\}$ has the property \eqref{property:ball;included}. The letter $C$ (with subscripts) will be used to denote various constants, not necessarily the same from line to line, which depend only on the quasi-metric constant $A_0$, the geometric doubling constant $A_1$ and the parameter $\delta$, but not on points, sets or functions considered. Such constants we refer to as \textit{geometric constants}.

\begin{lemma}\label{lem:existsdyadiccube}
Given $t=1,2,\ldots ,L$, $x\in X$ and $y\in X$, there exists $k\in \Z$ such that $y\in Q^{k}(x,t)$. Moreover, if $\rho(x,y)\geq\delta^k$, then $y\notin Q^{k+1}(x,t)$. In particular, if $\rho(x,y)>0$, there do not exist arbitrarily large indices $k$ such that $y\in Q^k(x,t)$.
\end{lemma}

\begin{proof}
Consider cubes $Q^k_\alpha\in\mathscr{D}^t$ as in \eqref{eq:fixed_point} which have $x_0$ as their center point. Pick $k\in \Z$ such that $x,y\in B(x_0,c_1\delta^{k})$. The first assertion follows from \eqref{eq:contain}. 

For the second assertion, suppose $\rho(x,y)\geq \delta^{k}$. 
Denote by $z^{k+1}_\alpha$ the center point of $Q^{k+1}(x,t)$. Then
\begin{align*}
\rho(y,z^{k+1}_\alpha) & \geq A_0^{-1}\rho(x,y)-\rho(x,z^{k+1}_\alpha)\geq A_0^{-1}\delta^{k}-C_1\delta^{k+1} > C_1\delta^{k+1}
\end{align*}
since $96A_0^2\delta <1$, showing that $y\notin Q^{k+1}(x,t)$.
\end{proof}

\begin{lemma}\label{lem:positivemeasure}
Suppose $\sigma$ and $\omega$ are non-trivial positive Borel-measures on $X$, and let $A\subseteq X$ be a measurable set with $\omega(A)>0$. For every $t=1,\ldots ,L$ there exists a dyadic cube $Q\in \mathscr{D}^t$ such that $\sigma (Q)>0$ and $\omega (A\cap Q)>0$.
\end{lemma}

\begin{proof}
For $k\in \Z$, consider the sets $B_k:=B(x_0,c_1\delta^{-k})$ and $A_k:=A\cap B_k$. First observe that $\sigma(B_k)>0$ for $k>k_0$ and $\omega(A_k)>0$ for $k>k_1$. Indeed, $X=\cup_{k=1}^{\infty}B_k$ and $B_1\subseteq B_2\subseteq \ldots$, so that $0<\sigma(X)=\lim_{k\to\infty}\sigma(B_k)$. Similarly, $A=\cup_{k=1}^{\infty}A_k$ and $A_1\subseteq A_2\subseteq \ldots$, so that $0<\omega(A)=\lim_{k\to\infty}\omega(A_k)$. Set $k=\max\{k_0,k_1\}$ and let $Q\in\mathscr{D}^t$ be the dyadic cube of generation $-k$ centred at $x_0$. Then $B_k\subseteq Q$ by \eqref{eq:contain}, and it follows that $\sigma(Q)\geq \sigma(B_k)>0$ and $\omega(A\cap Q)\geq \omega(A\cap B_k)=\omega(A_k)>0$.
\end{proof}

\begin{remark}
The proofs of the preceding two lemmata rest on the property~\eqref{eq:fixed_point}. Note that the two Lemmata are not in general true for the usual Euclidean dyadic cubes of the type
\[2^{-k}([0,1)^n+m),\quad k\in \Z, \, m\in \Z^n.
\] 
(E.g. consider $X=\R$ with $\sigma$ and $\omega$ the one-dimensional Lebesgue measures on $(-\infty ,0)$ and $(0,\infty)$, respectively, and the usual dyadic intervals. Then there exists no dyadic interval which intersects the supports of both $\sigma$ and $\omega$. Further, if $y\in (-\infty ,0)$ and $x\in (0,+\infty)$, then $y\notin Q^k(x)$ for all $k\in \Z$ when $Q^k(x)=2^{-k}[m,m+1), m\in \{0,1,2,\ldots\}$, is the dyadic interval of level $k$ that contains $x$.)
\end{remark}

We shall need the following elementary covering lemma.
\begin{definition}
Let $\mathscr{Q}$ be any collection of dyadic cubes. Then $Q=Q^k_\alpha\in \mathscr{Q}$ is \textit{maximal} (relative to the collection $\mathscr{Q}$) if for every $Q^\ell_\beta\in\mathscr{Q}$, $Q^\ell_\beta\cap Q^k_\alpha\neq \emptyset$ implies $\ell\geq k$.
\end{definition}

\begin{lemma}\label{lem:maximalcube}
Suppose $\mathscr{Q}\subseteq \mathscr{D}^t$ is a collection of dyadic cubes $Q=Q^k_\alpha$ restricted to $k\geq k_0$. 
Then every cube in $\mathscr{Q}$ is contained in a maximal cube and the maximal cubes are mutually disjoint.
\end{lemma}

We end this section with a proposition which we find interesting. As an application of the Proposition, we will show that for a potential type operator $T$, the testing condition
\[
\left\| \chi_{Q}T(\chi_{Q}\, d\sigma)\right\|_{L^q_\omega}\leq C\sigma(Q)^{1/p} \quad\text{for all dyadic cubes } Q\in\bigcup_{t=1}^L \mathscr{D}^t\]
implies the qualitative property that $T(\chi_B d\sigma)\in L^q_\omega$ for all balls $B$. Originally, the proof of this testing type result \cite[pp.549-552]{SWZ96} required a group structure on $X$. In fact, the group structure allows the translations of the dyadic lattice leading to the existence of a sequence of dyadic cubes as in the following lemma. 

\begin{proposition}\label{lem:sequence_of_cubes}
Given a ball $B=B(x_B,r_B)$, there exists a sequence $(Q_k)_{k\geq 1}\subseteq \bigcup_{t=1}^{L}\mathscr{D}^t$ of dyadic cubes (possibly from different systems) with the properties that
\begin{equation*}\label{prop;sequence}
\begin{split} 
\text{(i)} & \quad
B\subseteq Q_1\subseteq Q_2\subseteq \ldots \subseteq Q_k \subseteq Q_{k+1}\subseteq \ldots ;\\
\text{(ii)} & \quad \text{There exists a geometric constant $c_0>1$ such that $\diam Q_k\leq c_0^k r_B$ for every $k\geq 1$}; \\
\text{(iii)} & \quad c_0^{k-1}B\subseteq Q_k\subseteq c_0^k B \;\text{for every $k\geq 1$}.
\end{split}
\end{equation*}
\end{proposition}

We will use the adjacent dyadic systems to construct the sequence without assuming a group structure, as stated. In \cite{SWZ96}, only the existence of such a sequence and not the proof of the mentioned testing type result require the group structure; the testing result itself is stated and proved in a Sawyer--Wheeden type space described in \ref{def:SWspace}. However, the proof does not use the non-empty annuli property nor depend on the assumption imposed on the measures having no atoms, but only uses the properties of the sequence provided by Proposition~\ref{lem:sequence_of_cubes}, the properties \eqref{kernel} of the kernel, the positivity of the operator and the assumed testing condition. Hence, the proof applies in the present context, and by assuming Proposition~\ref{lem:sequence_of_cubes}, we may state the following Lemma and refer to the original proof given in \cite[pp. 549--552]{SWZ96}. 

\begin{lemma}\label{lem:finite_for_bounded}
Assume that the $(Q,Q,Q)$ testing condition
\begin{equation}\label{remark;hypothesesone}
\left(\int_{Q}T(\chi_{Q}\, d\sigma)^q\, d\omega\right)^{1/q}\leq C\sigma (Q)^{1/p}
\end{equation}
holds for every dyadic cube $Q\in \bigcup_{t=1}^{L}\mathscr{D}^t$. Then $T(\chi_B\, d\sigma)\in L^q(X,\omega)$ for all balls $B$. As a consequence, $T(f\, d\sigma)\in L^q(X,\omega)$ for all bounded $f$ with bounded support.
\end{lemma}

\begin{proof}[Proof of Proposition~\ref{lem:sequence_of_cubes}]
Recall from \eqref{property:ball;included} that given a ball $B=B(x,r)$, there exists a dyadic cube $Q=:Q_1\in \cup_{t=1}^{L}\mathscr{D}^t$ such that 
\[B\subseteq Q_1,\quad \diam Q_1\leq c_0r, \quad c_0=8A_0^3\delta^{-2}>1. \]
In particular, for every $y\in Q_1$ we have that $\rho(x,y)\leq c_0r$ and consequently, $Q_1\subseteq c_0B$. 

Next consider the ball $c_0B$. By repeating the reasoning made above, we find a dyadic cube $Q_2$ (possibly from some other dyadic system than $Q_1$) such that 
\[c_0B\subseteq Q_2,\quad \diam Q_2\leq c_0^2r, \quad Q_2\subseteq c_0^2B. \] 
By iteration, this yields a sequence $(Q_k)_{k\geq 1}$ with the desired properties. 
\end{proof}

\section{The dyadic model of $T$}\label{sec:dyadic}
Let $T$ be a potential type operator defined in \ref{def:PO}. We shall tacitly assume that the kernel satisfies $K(x,y)<\infty$ for $x\neq y$, and that the functions $K(\cdot,y)\colon (X,\sigma)\to [0,\infty]$ and $K(x,\cdot)\colon (X,\sigma)\to [0,\infty]$ are measurable for fixed $x,y\in X$, and further, that the integral \eqref{def:operatorT} defines a measurable function $(X,\omega)\to [0,\infty]$. We also consider the (formal) adjoint $T^\ast$ of $T$, defined by
\[T^\ast(g\,d\omega)(y)=\int_{X}K(x,y)g(x)\,d\omega(x),\quad y\in X.\]

Our investigations continue the earlier work of Sawyer and Wheeden \cite{SW92}, Sawyer, Wheeden and Zhao \cite{SWZ96} and Verbitsky and Wheeden \cite{VW98}. 

\begin{remark}
In the earlier papers on the topic \cite{SW92, SWZ96, VW98}, the set-up is a Sawyer-Wheeden type space described in Definition~\ref{def:SWspace}. In particular, it is there assumed that all the annuli $B(x,R)\setminus B(x,r)$ are non-empty for $x\in X$ and $0<r<R<\infty$. We mention that having this non-empty annuli property, if the growth conditions in \eqref{kernel} hold for some $k_1>1$ and $k_2>1$, then for any $k_2'>1$ there exists $k_1'>1$ such that \eqref{kernel} holds with $k_1$ and $k_2$ replaced by $k_1'$ and $k_2'$, respectively. 
\end{remark}

We will show that operator $T$ has a dyadic version $T^{\mathscr{D}^t_{\sigma\omega}}$ (to be defined below in \ref{def:dyadicoperators}) associated to each dyadic system $\mathscr{D}^t$ and the measures $\sigma$ and $\omega$, and provide in Lemmata \ref{lem:eqv1;T} and \ref{lem:eqv2;T} below the following pointwise equivalence between the original operator and its dyadic counterparts:

\begin{proposition}\label{prop:TheoremC}
We have the pointwise estimates
\begin{displaymath}
T^{\mathscr{D}^t_{\sigma\omega}}(f\, d\sigma)(x) \leq C\,\left\{ \begin{array}{l}
T(f\, d\sigma)(x) \\
T^\ast(f\, d\sigma)(x)
\end{array} \right.
\quad\text{and}\quad T(f\,d\sigma)(x)\leq C\sum_{t=1}^{L}T^{\mathscr{D}^t_{\sigma\omega}}(f\, d\sigma)(x).
\end{displaymath}
The constant $C>0$ is geometric (independent of $x$ and $f$). The inequalities on the left hold for all $x\in X$ and $t=1,\ldots ,L$, and the inequality on the right for $\omega$-a.e. $x\in X$. 
\end{proposition}

\begin{remark}
Minkowski's inequality (the triangle inequality for $L^q$-norms) together with the inequalities on the right of Proposition~\ref{prop:TheoremC} imply that
\[\|T(f\,d\sigma) \|_{L^q_\omega}\leq C \sum_{t=1}^{L}\|T^{\mathscr{D}^t_{\sigma\omega}}(f\,d\sigma) \|_{L^q_\omega}. \] 
This sort of estimate in norm was proven in \cite[Theorem 1.1]{VW98}, cf. (\cite[Lemma 4.7]{SW92} and \cite[Lemma 3.1]{SWZ96}), with the summation on the right replaced by a supremum over all translations of dyadic cubes. To allow the translations, it was assumed that $X$ supports a doubling measure and has a related group structure. Our result sharpens these previous results; we give a pointwise estimate without an underlying doubling measure or a group structure.
\end{remark}

\subsection{Preparations}
As a preparation for constructing a dyadic model of $T$, we define a set function $\varphi$, cf. \cite{SW92}, related to the kernel $K$: For a dyadic cube $Q$, we set
\begin{equation}\label{varphi;setfunction}
\varphi (Q)=\varphi_K(Q):=\sup \lbrace K(x,y)\colon x,y\in B(Q), \, \rho(x,y)\geq cr_{B(Q)}\rbrace  \in [0,\infty],
\end{equation}
where $c:=\delta^2/(5A_0^2)\in (0,1)$ is a small geometric constant, $B(Q)$ is the containing ball of $Q$ as in \eqref{eq:contain} and $r_{B(Q)}$ is the radius of $B(Q)$. We agree that $\varphi (Q)=0$ if the points in the definition \eqref{varphi;setfunction} do not exist. In fact, $\varphi(Q)<\infty$. Moreover, the properties \eqref{kernel} on the kernel lead to useful growth estimates for $\varphi$:

\begin{lemma}[Kernel estimates]\label{lemma;varphi}
There exists a geometric constant $C\geq 1$ such that for a dyadic cube $Q\in\mathscr{D}^t$,
\begin{enumerate}
\item[\textit{(i)}] $\varphi (Q)\leq CK(x,y)$ for all $x,y\in B(Q)$. 
In particular, $\varphi(Q)<\infty$ and $\varphi (Q)\leq CK(x,x)$ for $x\in Q$;
\item[\textit{(ii)}] If $P\in\mathscr{D}^t$ is a dyadic cube such that $P\subseteq Q$ and $\lbrace x,y\in B(P)\colon \rho(x,y)\geq cr_{B(P)}\rbrace\neq \emptyset$, then
\[\varphi (Q)\leq C\varphi (P).  \]
\end{enumerate}
Moreover,

\textit{(iii)} If $\lbrace x,y\in B(Q^k_\alpha), \, \rho(x,y)\geq cr_{B(Q^k_\alpha)}\rbrace = \emptyset$ and $Q^{k+1}_\beta\subseteq Q^k_\alpha$, then $Q^k_\alpha = Q^{k+1}_\beta$.

\end{lemma}
We mention that the properties $(i)$ and $(ii)$ were pointed out by Sawyer, Wheeden and Zhao \cite[formulae (4.1) and (4.3)]{SWZ96} where it was additionally assumed that all annuli $B(x,R)\setminus B(x,r)$ are non-empty for $0<r<R$ and $x\in X$. With this extra assumption, $(ii)$ of the Lemma holds for all $P\subseteq Q$, and $(iii)$ does not occur. 

\begin{proof}
Fix a dyadic cube $Q$. Consider points $x^\ast , y^\ast\in B(Q)$ with $\rho(x^\ast,y^\ast)\geq cr_{B(Q)}$. (If no such points exist $\varphi (Q)=0$, and $(i)$ and $(ii)$ follow.) 
Let $x,y\in B(Q)$. Then $\rho(x,y)\leq 2A_0r_{B(Q)}$. We will show that $K(x^\ast,y^\ast)\leq CK(x,y)$ with a constant $C$ independent of $Q$ and the points. 

We may assume $\rho (x^\ast ,y)\geq \rho(y^\ast ,y)$ (since in case $\rho (y^\ast ,y)\geq \rho(x^\ast ,y)$ we argue similarly with the roles of $x^\ast$ and $y^\ast$ interchanged). Thus, $cr_{B(Q)}\leq \rho(x^\ast ,y^\ast)\leq A_0(\rho(x^\ast ,y)+\rho(y ,y^\ast))\leq 2A_0\rho (x^\ast ,y)$, and consequently, $r_{B(Q)}\leq 2A_0c^{-1}\rho (x^\ast ,y)$. Hence, $\rho(x,y)\leq 2A_0r_{B(Q)}\leq 4A_0^2c^{-1}\rho(x^\ast ,y)$. By the growth conditions \eqref{kernel} imposed on the kernel $K$, this implies that 
\[K(x^\ast ,y)\leq k_1K(x,y)\] 
with some $k_1>1$ depending only on the geometric $k_2:=4A_0^2c^{-1}=20A_0^4/\delta^2>1$ (and kernel $K$). On the other hand, also $\rho(x^\ast,y)\leq 2A_0r_{B(Q)}\leq 4A_0^2c^{-1}\rho(x^\ast ,y^\ast)$ implying, again by \eqref{kernel}, that 
\[K(x^\ast ,y^\ast )\leq k_1K(x^\ast,y).\] 
We conclude with $K(x^\ast ,y^\ast)\leq k_1^2K(x,y)$. By the definition of $\varphi$, $(i)$ follows. 

For the second assertion, consider a dyadic cube $P\subseteq Q$ in $\mathscr{D}^t$, and let $x,y\in B(P)$. Recall from \eqref{eq:monotone} that $P\subseteq Q$ implies $B(P)\subseteq B(Q)$. In particular $x,y\in B(Q)$, and $(i)$ implies that $\varphi(Q)\leq CK(x,y)$. If $\lbrace x,y\in B(P)\colon \rho(x,y)\geq cr_{B(P)}\rbrace\neq \emptyset$, $(ii)$ follows by the definition of $\varphi$.

For the third assertion, suppose that $\lbrace x,y\in B(Q^k_\alpha), \, \rho(x,y)\geq cr_{B(Q^k_\alpha)}\rbrace = \emptyset$. Recall from \eqref{eq:contain} that $r_{B(Q^k_\alpha)} = C_1\delta^k=4A_0^2\delta^k$. Thus,
\begin{equation}\label{eq:smalldistance}
\rho(x,y)< cr_{B(Q^k_\alpha)} = \delta^2/(5A_0^2)\cdot 4A_0^2\delta^k <\delta^{k+2} \text{ for all } x,y\in B(Q^k_\alpha).
\end{equation}
Suppose that $Q^{k+1}_\beta\subseteq Q^k_\alpha$. By \eqref{eq:smalldistance}, there in particular holds that for all $y\in Q^k_\alpha$ we have $\rho(z^{k+1}_\beta,y)< \delta^{k+2}$ implying that $Q^k_\alpha\subseteq B(z^{k+1}_\beta,\delta^{k+2})\subseteq B(z^{k+1}_\beta,c_1\delta^{k+1})\subseteq Q^{k+1}_\beta$ since $96A_0^6\delta\leq 1$ and $c_1=(12A_0^4)^{-1}$, and by \eqref{eq:contain}. This shows that $Q^{k+1}_\beta =Q^k_\alpha$.
\end{proof}

\subsection{Generalised dyadic cubes}\label{sec:generalisedcubes}
In our investigations, it is convenient to slightly enlarge the set of dyadic cubes. This will provide a tool for treating the prospective atoms. To this end, for a positive Borel-measure $\sigma$, denote 
\[X_\sigma:=\{x\in X\colon \sigma(\{x\})>0\},\]
the set of $\sigma$-atoms in $X$. Note that under the assumption that $\sigma(B)<\infty$ for all balls $B$, the set $X_\sigma$ is at most countable. 

Given positive Borel-measures $\sigma$ and $\omega$ which have the property that $\sigma(B)<\infty$ and $\omega(B)<\infty$ for all balls $B$, we denote by
\[X_{\sigma\omega}:=X_\sigma\cap X_\omega\] 
the set of joint atoms. For every $t=1,\ldots, L$, we declare that
\[\mathscr{D}^t_{\sigma\omega} := \mathscr{D}^t\cup \left(\bigcup_{x\in X_{\sigma\omega}}\{x\}\right). \]
We will refer to the elements in $\mathscr{D}^t_{\sigma\omega}$ as \textit{(generalised) dyadic cubes}. The elements in $\mathscr{D}^t$, which are independent of measures, are then called \textit{standard dyadic cubes} and the elements $\{x\}, x\in X_\sigma\cap X_\omega$, depending on measures, are called \textit{point cubes}. Note that there might happen that for some $Q=Q^k_\alpha\in\mathscr{D}^t$ we have $Q^k_\alpha = \{z^k_\alpha \}$, and $z^k_\alpha\in X_{\sigma\omega}$. In such a case, $Q$ will be treated as a standard dyadic cube.

\begin{remark}
Suppose $(X,\rho,\mu)$ is a space of homogeneous type. In the special case that the two measures $\sigma$ and $\omega$ are both absolutely continues with respect to an underlying doubling measure $\mu$, we have that $X_{\sigma\omega}\subseteq \mathscr{D}^t$ and thus, $\mathscr{D}^t_{\sigma\omega}=\mathscr{D}^t$. Indeed, consider $x\in X_{\sigma\omega}$. Then $\mu(\{x\})>0$. It is well-known that this implies $\{x\}=B(x,\varepsilon)$ for some $\varepsilon>0$ for doubling $\mu$. Pick $k_0\in\Z$ such that $8A_0^3\delta^{k_0}\leq \varepsilon$, and let $k>k_0$. Since $x\in Q^k_\alpha$ for some $\alpha$ by \eqref{eq:cover}, we have $\rho(x,z^k_\alpha)< 4A_0^2\delta^k$ by \eqref{eq:contain}. This implies that $Q^k_\alpha\subseteq B(z^k_\alpha,4A_0^2\delta^k)\subseteq B(x,\varepsilon)=\{x\}$ and thereby $\{x\}=Q^k_\alpha$. 
\end{remark}

The following lemma indicates the fact that the testing conditions for standard dyadic cubes imply the same conditions for point cubes.

\begin{lemma} 
Let $T$ be a potential type operator and assume that $T$ satisfies the testing condition 
\[[\sigma ,\omega]_{S_{p,q}}=\sup_{Q\in \mathscr{D}^t}\sigma(Q)^{-1/p}
\left\| \chi_{Q}T(\chi_{Q}\, d\sigma)\right\|_{L^q_\omega}<\infty . 
\]
Then for all $x\in X_{\sigma\omega}$,
\begin{equation}\label{eq:pointcubes}
 \sigma(\{ x\})^{-1/p}\|\chi_{\{x \}}T(\chi_{\{x \}}\, d\sigma)\|_{L^q_\omega} \leq [\sigma ,\omega]_{S_{p,q}} <\infty .
 \end{equation}
In particular, we have the testing inequality
\[\left\| \chi_{Q}T(\chi_{Q}\, d\sigma)\right\|_{L^q_\omega}\leq [\sigma ,\omega]_{S_{p,q}} \sigma(Q)^{1/p}\]
for all $Q\in\mathscr{D}^t_{\sigma\omega}$. 
\end{lemma}

\begin{proof}
Suppose $x\in X_{\sigma\omega}$, and consider the sequence $(Q^k)_{k\geq 1}\subseteq\mathscr{D}^t$ of nested dyadic cubes $Q^k=Q^k(x,t)$ (of generations $k\geq 1$) which shrinks to $x$. Then for all $k$,
\[\|\chi_{\{x \}}T(\chi_{\{x \}}\, d\sigma)\|_{L^q_\omega}\leq \|\chi_{Q^k}T(\chi_{Q^k}\, d\sigma)\|_{L^q_\omega}
\leq \sigma(Q^k)^{1/p}[\sigma,\omega]_{S_{p,q}} ,\]
so that
\[\left(\frac{\sigma(\{x\})}{\sigma(Q^k)}\right)^{1/p}
\sigma(\{x\})^{-1/p}\|\chi_{\{x \}}T(\chi_{\{x \}}\, d\sigma)\|_{L^q_\omega}\leq [\sigma,\omega]_{S_{p,q}}.\]
The claim follows by observing that $\sigma(Q^k)\to \sigma(\{x\})$ as $k\to\infty$.
\end{proof}

\begin{remark}
Similarly, the dual testing condition
\[[\omega ,\sigma]^\ast_{S_{q',p'}}=\sup_{Q\in\mathscr{D}^t}\omega(Q)^{-1/q'}
\left\| \chi_{Q}T^\ast(\chi_{Q}\, d\omega)\right\|_{L^{p'}_\sigma}<\infty  \]
implies that for all $x\in X_{\sigma\omega}$,
\begin{equation}\label{eq:pointcubes2}
 \omega(\{ x\})^{-1/q'}\|\chi_{\{x \}}T^\ast(\chi_{\{x \}}\, d\omega)\|_{L^{p'}_\sigma} 
 \leq [\omega ,\sigma]^\ast_{S_{q',p'}} 
 <\infty .
 \end{equation}
In conclusion, if $x\in X_{\sigma\omega}$ and $T$ satisfies both the testing condition, then \eqref{eq:pointcubes} and \eqref{eq:pointcubes2} hold.
Note that \eqref{eq:pointcubes} and \eqref{eq:pointcubes2} are equivalent to
\[K(x,x)\omega(\{x\})^{1/q}\leq [\sigma ,\omega]_{S_{p,q}}\sigma(\{ x\})^{1/p-1} \quad \text{and} \quad 
K(x,x)\sigma(\{x\})^{1/p'}\leq [\omega ,\sigma]^\ast_{S_{q',p'}}\omega(\{ x\})^{1/q'-1} ,\]
respectively. Thus, the testing conditions in particular imply that in case $\sigma$ and $\omega$ have a joint atom at $x$, then the kernel $K$ must satisfy $K(x,x)<\infty$.
\end{remark}

\subsection{Dyadic model operators}\label{def:dyadicoperators} 
We will associate to each family $\mathscr{D}^t_{\sigma\omega}, t=1,\ldots ,L$, of generalised dyadic cubes, a dyadic model operator $T^{\mathscr{D}^t_{\sigma\omega}}$ defined as follows. 

For a dyadic cube $Q\in \mathscr{D}^t$, we denote by $Q^{(1)}$ the (unique) dyadic parent of $Q$ (i.e. the next larger cube in $\mathscr{D}^t$ that contains $Q$). Also recall the notation $Q^k(x,t)$ for the (unique) dyadic cube in $\mathscr{D}^t$ of generation $k$ which contains $x\in X$. Of course, $Q^k(x,t)$ always depends on $t$ and $x$ but we may omit one or both of these dependences in the notation whenever they are clear from the context. We define the dyadic version of $T(f\, d\sigma)$ associated to the family $\mathscr{D}^t_{\sigma\omega}$ by
\begin{align*}
T^{\mathscr{D}^t_{\sigma\omega}}(f\, d\sigma)(x) &:= 
\sum_{Q\in \mathscr{D}^t}\chi_{Q}(x)\varphi(Q^{(1)})\int_{Q^{(1)}\setminus Q}f(y)\,d\sigma(y) + \sum_{z\in X_{\sigma\omega}}\chi_{\{z\}}(x)K(z,z)f(z)\sigma(\{z\})\\
&=\sum_{k\in \Z} \varphi (Q^k(x,t))\int_{Q^k(x,t)\setminus Q^{k+1}(x,t)}f(y)\, d\sigma (y) +\chi_{X_{\sigma\omega}}(x)K(x,x)f(x)\sigma(\{ x \}), \; f\geq 0.
\end{align*}

This more tractable dyadic model operator was introduced and investigated by Verbitsky and Wheeden \cite{VW98}. Sawyer and Wheeden \cite{SW92}, later Sawyer, Wheeden and Zhao \cite{SWZ96} and very recently Lacey, Sawyer and Uriarte--Tuero \cite{LSUT:09} studied a closely related pointwise larger dyadic operator $T_{\mathcal{G}}$ formed by integrating over all of $Q^k(x,t)$ instead of just $Q^k(x,t)\setminus Q^{k+1}(x,t)$. The larger function $T_{\mathcal{G}}(f \, d\sigma)$ is similarly related to $T(f\, d\sigma)$, but the pointwise estimate $T_{\mathcal{G}}(f \, d\sigma)(x)\leq CT(f\, d\sigma)(x)$; cf. Lemma~\ref{lem:eqv1;T} below, is only known to hold under an extra hypothesis imposed on the kernel $K$, namely \cite[formula (1.24)]{SW92}: For some $\epsilon>0$,
\[\varphi(B)\leq C\left(\frac{r(B')}{r(B)}\right)^\epsilon\, \varphi(B')\quad\text{for all balls }B'\subseteq 2A_0B. \]

\bigskip

The dyadic operator $T^{\mathscr{D}^t_{\sigma\omega}}$ has a symmetric kernel while this is not necessarily the case for $T$:
\begin{lemma}\label{rem:symmetrickernel}
The operator $T^{\mathscr{D}^t_{\sigma\omega}}$ can be presented as
\[T^{\mathscr{D}^t_{\sigma\omega}}(f\, d\sigma)(x) = \int_{X}k(x,y)f(y)\, d\sigma(y). \]
Kernel $k$ is the positive measurable function 
\[k(x,y)=\begin{cases} 
\varphi(Q(x,y))&  \text{when } x\neq y;\\
\chi_{X_{\sigma\omega}}(x)K(x,x) & \text{when } x=y, 
\end{cases} \] 
where $Q(x,y), x\neq y$, is the smallest dyadic cube in $\mathscr{D}^t$ that contains both $x$ and $y$. 
\end{lemma}

\begin{proof}
Let $x\in X$, and write
\begin{equation*}
\begin{split}
T^{\mathscr{D}^t_{\sigma\omega}}(f\, d\sigma)(x) & = \sum_{k\in \Z}\varphi (Q^k)\int_{X\setminus \{x\}}f(y)\chi_{Q^k\setminus Q^{k+1}}(y)\, d\sigma(y)
+\chi_{X_{\sigma\omega}}(x)K(x,x)f(x)\sigma(\{ x \})  \\
&= \int_{X\setminus \{x\}}\left(\sum_{k\in \Z}\varphi (Q^k)\chi_{Q^k\setminus Q^{k+1}}(y)\right)f(y)\, d\sigma(y) +\chi_{X_{\sigma\omega}}(x)K(x,x)f(x)\sigma(\{ x \}) .
\end{split}
\end{equation*}
Momentarily, fix $y\in X\setminus \{ x\}$. Recall from Lemma~\ref{lem:existsdyadiccube} that there exists $k_0\in \Z$ such that $y\in Q^{k_0}(x)$ and that there do not exist arbitrarily large such indices (hence, arbitrarily small cubes). Let $l\geq k_0$ be the largest index such that $y\in Q^{l}(x)$. Then $\chi_{Q^k(x)\setminus Q^{k+1}(x)}(y)=0$ for every $k>l$. By nestedness, $\chi_{Q^k(x)\setminus Q^{k+1}(x)}(y)=0$ for every $k<l$. It follows that 
\[\sum_{k\in \Z}\varphi (Q^k(x))\chi_{Q^k(x)\setminus Q^{k+1}(x)}(y) = \varphi (Q^l(x)), \]
where $Q^l(x)=:Q(x,y)$ is the smallest dyadic cube containing both $x$ and $y$. As a consequence, 
\begin{align*}
T^{\mathscr{D}^t_{\sigma\omega}}(f\, d\sigma)(x) & =
\int_{X\setminus \{x\}}\varphi (Q(x,y))f(y)\, d\sigma(y)
+\chi_{\sigma\omega}(x)K(x,x)f(x)\sigma(\{ x \})\\
& = \int_{X}k(x,y)f(y)\, d\sigma(y) 
\end{align*}
with the kernel $k(x,y)=\varphi (Q(x,y))$, $x\neq y$, and $k(x,x)=\chi_{\sigma\omega}(x)K(x,x)$, as claimed. 
\end{proof}

\subsection{Duality}\label{def:duality} 
By the symmetry of the dyadic kernel $k$, indicated by Lemma~\ref{rem:symmetrickernel}, we have the following duality identity for any measurable $g,h\geq 0$:
\begin{align*}
\langle T^{\mathscr{D}^t_{\sigma\omega}}(g\,d\sigma),h \rangle^{\omega} & :=\int_{X}T^{\mathscr{D}^t_{\sigma\omega}}(g\, d\sigma)(x)h(x)\,d\omega(x)  =
\int_{X}\left(\int_{X}k(x,y)g(y)\,d\sigma(y)\right)h(x)\,d\omega(x)\\
& =\int_{X}g(y)\left(\int_{X }k(y,x)h(x)\,d\omega(x) \right)d\sigma(y)\quad\text{by Fubini's}\\
& = \int_{X}g(y)\, T^{\mathscr{D}^t_{\sigma\omega}}(h\, d\omega)(y)\,d\sigma(y)=\langle g,T^{\mathscr{D}^t_{\sigma\omega}}(h\,d\omega) \rangle^{\sigma}.
\end{align*}
This shows that the dyadic operator $T^{\mathscr{D}^t_{\sigma\omega}}$ is self-adjoint.

\bigskip

In proofs we will need minor technical variants of $T^{\mathscr{D}^t_{\sigma\omega}}$ introduced by Verbitsky and Wheeden \cite{VW98}: For a fixed positive integer $m$, we define 
\[T^{\mathscr{D}^t_{\sigma\omega}}_{m}\!(f\, d\sigma)(x):=\sum_{k\in \Z}\varphi (Q^k(x))\int_{Q^k(x)\setminus Q^{k+m}(x)}f\, d\sigma 
+\chi_{X_{\sigma\omega}}(x)K(x,x)f(x)\sigma(\{x \}),  \quad f\geq 0. \] 
Note that with $m=1$, we have $T^{\mathscr{D}^t_{\sigma\omega}}_{1}=T^{\mathscr{D}^t_{\sigma\omega}}$. 

We will record the following equivalence between the dyadic model operator and its modifications. The estimates are technical conclusions which will be needed when proving Lemma~\ref{lem:eqv2;T} below. We mention that the following lemma is proved in \cite[Lemma 2.1]{VW98} assuming that $\sigma(\{ x\})=0$ for all $x\in X$ and that all annuli $B(x,R)\setminus B(x,r)$ are non-empty for $0<r<R$ and $x\in X$. 
\begin{lemma}\label{lem;equivalense;dyad;operators}
For every $x\in X$ and positive integer $m$,
\[T^{\mathscr{D}^t_{\sigma\omega}}(f\, d\sigma)(x)\leq T^{\mathscr{D}^t_{\sigma\omega}}_{m}(f\, d\sigma)(x) \leq Cm\, T^{\mathscr{D}^t_{\sigma\omega}}(f\, d\sigma)(x).\]
The constant $C>0$ is geometric (independent of $x,m$ and $f$).
\end{lemma}

\begin{proof}
Fix a positive integer $m$ and $x\in X$, and consider the cubes $Q^k:=Q^k(x,t)\in\mathscr{D}^t, k\in \Z$. First note that the term $\chi_{X_{\sigma\omega}}(x)K(x,x)f(x)\sigma(\{x \})$ appears in the definition of both the operators $T^{\mathscr{D}^t_{\sigma\omega}}$ and $T^{\mathscr{D}^t_{\sigma\omega}}_m$ so that it suffices to only consider the  ``standard cube parts'' of the two operators, and we may assume that $\sigma(\{ x\})=0$. By the nestedness property $Q^{i}\subseteq Q^k$ for $i\geq k$, we have the inclusion $Q^k\setminus Q^{k+1}\subseteq Q^k\setminus Q^{k+m}$. Thus
\[\int_{Q^k\setminus Q^{k+1}}f\, d\sigma\leq \int_{Q^k\setminus Q^{k+m}}f\, d\sigma \]
for each $k$ and $f\geq 0$, so that 
\begin{align*}
T^{\mathscr{D}^t_{\sigma\omega}}(f\, d\sigma)(x) & =\sum_{k\in \Z}\varphi (Q^k)\int_{Q^k\setminus Q^{k+1}}f\, d\sigma \leq \sum_{k\in \Z}\varphi (Q^k)\int_{Q^k\setminus Q^{k+m}}f\, d\sigma  \\
& = T^{\mathscr{D}^t_{\sigma\omega}}_m(f\, d\sigma)(x). 
\end{align*}
For the reverse inequality, we write for each $k$,
\[Q^k\setminus Q^{k+m}=\bigcup_{i=k}^{k+m-1}\left(Q^i\setminus Q^{i+1}\right)\quad\text{(disjoint union)}, \]
and accordingly
\[\varphi(Q^k)\int_{Q^k\setminus Q^{k+m}}f\, d\sigma = 
\varphi(Q^k)\sum_{i=k}^{k+m-1}\int_{Q^i\setminus Q^{i+1}}f\, d\sigma . \]
We use the kernel estimates of Lemma~\ref{lemma;varphi} as follows: if for some $i$ in the sum above, $\{x,y\in B(Q^i)\colon \rho(x,y)\geq cr_{B(Q^i)}\}=\emptyset$ and hence $\varphi(Q^i)=0$, then $Q^i\setminus Q^{i+1}=\emptyset$ by Lemma~\ref{lemma;varphi}(iii), and the related term vanishes. Thus, by Lemma~\ref{lemma;varphi}(ii), we have an estimate
\[\varphi (Q^k)\int_{Q^i\setminus Q^{i+1}}f\, d\sigma \leq
C\varphi(Q^i)\int_{Q^i\setminus Q^{i+1}}f\, d\sigma \]
for each $i=k,\ldots ,k+m-1$. This leads to
\[\varphi(Q^k)\int_{Q^k\setminus Q^{k+m}}f\, d\sigma  \leq C \sum_{i=k}^{k+m-1}\varphi(Q^i)\int_{Q^i\setminus Q^{i+1}}f\, d\sigma . \]
We sum over $k$ and change the order of summation to conclude with
\begin{equation*}
\begin{split}
T^{\mathscr{D}^t_{\sigma\omega}}_{m}\!(f\, d\sigma)(x) & =\sum_{k\in \Z}\varphi (Q^k)\int_{Q^k\setminus Q^{k+m}}f\, d\sigma   \leq C\sum_{k\in \Z} \left(\sum_{i=k}^{k+m-1}\varphi(Q^i)\int_{Q^i\setminus Q^{i+1}}f\, d\sigma\right)  \\
& \leq Cm\left(\sum_{k\in \Z}\varphi(Q^k)\int_{Q^k\setminus Q^{k+1}}f\, d\sigma \right) = Cm\, T^{\mathscr{D}^t_{\sigma\omega}}(f\, d\sigma)(x).
\end{split}
\end{equation*}
\end{proof}

The following two lemmata provide the key step in our proof for the main theorem~\ref{thm:theoremB}: 

\begin{lemma}\label{lem:eqv1;T}
For every $x\in X$ and $t=1,\ldots ,L$ we have the pointwise estimates
\begin{displaymath}
T^{\mathscr{D}^t_{\sigma\omega}}(f\, d\sigma)(x) \leq C\,\left\{ \begin{array}{l}
T(f\, d\sigma)(x) \\
T^\ast(f\, d\sigma)(x).
\end{array} \right.
\end{displaymath}
The constant $C>0$ is geometric (independent of $x$ and $f$). 
\end{lemma}

\begin{lemma}\label{lem:eqv2;T}
For $\omega$-a.e. $x\in X$ we have the pointwise estimate
\[T(f\,d\sigma)(x)\leq C\sum_{t=1}^{L}T^{\mathscr{D}^t_{\sigma\omega}}(f\, d\sigma)(x).\]
The constant $C>0$ is geometric (independent of $x$ and $f$). 
\end{lemma}

\begin{remark}
We make the elementary observation that
\[T(f\, d\sigma)(x)=\int_{X\setminus \{x \}}K(x,y)f(y)\,d\sigma(y)+K(x,x)f(x)\sigma(\{x \}) \]
so that $T$ has two parts of which the latter one is ``dyadic'' in the sense that the non-negative term $K(x,x)f(x)\sigma(\{x \})$ also appears in the definition of the dyadic operators defined in \ref{def:dyadicoperators}. Note that if the term $K(x,x)f(x)\sigma(\{x \})$ contributes to $T^{\mathscr{D}^t_{\sigma\omega}}(f\, d\sigma)(x)$, then $\sigma(\{x\})>0$ so that it also contributes to $T(f\, d\sigma)(x)$. On the other hand, the set where the term $K(x,x)f(x)\sigma(\{x \})$ contributes to $T(f\, d\sigma)(x)$ but does not contribute to $T^{\mathscr{D}^t_{\sigma\omega}}(f\, d\sigma)(x)$ consists of points $x\in X$ with $\sigma(\{x\})>0$ and $\omega(\{x\})=0$, and this is an $\omega$-null set (recall that the set $X_\sigma$ is at most countable). Thus, in order to prove Lemmata~\ref{lem:eqv1;T}, \ref{lem:eqv2;T} we may assume that $\sigma(\{x\})=0$. Note that Lemmata~\ref{lem:eqv1;T} and \ref{lem:eqv2;T} complete the proof of Proposition~\ref{prop:TheoremC}.
\end{remark}

To prove Lemma~\ref{lem:eqv1;T}, we may (by recalling the kernel estimates of Lemma~\ref{lemma;varphi}) refer to the proof given in \cite[Lemma 2.2]{VW98}. 
Lemma~\ref{lem:eqv2;T}, however, is a new result. 

\begin{proof}[Proof of Lemma~\ref{lem:eqv2;T}]
Fix $x\in X$. We write (recall that we may assume $\sigma(\{x\})=0$)
\[T(f\,d\sigma)(x)=\sum_{\ell =-\infty}^{\infty}\int_{\lbrace y\in X\colon \delta^{\ell +1}\leq\rho (x,y)< \delta^\ell \rbrace} K(x,y)f(y)\, d\sigma (y)
. \]
Momentarily, fix $\ell\in \Z$ and consider $y\in X$ with $\delta^{\ell+1}\leq\rho (x,y)< \delta^\ell$. Recall from \eqref{property:ball;included} that there exists a dyadic system $\mathscr{D}^t$, $t=t(x,\ell)$, and a dyadic cube $Q^{\ell-1}\in\mathscr{D}^t$ such that $B(x,\delta^\ell)\subseteq Q^{\ell-1}$ (hence, $Q^{\ell-1}$ is the unique cube in $\mathscr{D}^t$ of generation $\ell-1$ which contains $x$). In particular, $y\in Q^{\ell-1}$ for each relevant $y$. Also recall, from \eqref{eq:contain}, that for the radius of the containing ball of $Q^{\ell-1}$ we have $r_{B(Q^{\ell-1})}=C_1\delta^{\ell-1}=4A_0^2\delta^{\ell -1}$. Also note that for the parameter $c$ in the definition of $\varphi$ in \eqref{varphi;setfunction} we have $c=:\delta^2/(5A_0^2)<\delta^2/C_1$.  Thus, $cr_{B(Q^{\ell-1})}<\delta^{\ell+1}$, and we obtain
\begin{equation*}
\begin{split}
\varphi (Q^{\ell -1}) &= \sup \lbrace K(y,y') \colon y,y'\in B(Q^{\ell -1}), \rho(y,y')\geq cr_{B(Q^{\ell-1})}\rbrace\\
&\geq \sup \lbrace K(y,y') \colon y,y'\in B(Q^{\ell -1}), \rho(y,y')\geq \delta^{\ell +1}\rbrace\\
&\geq K(x,y)\quad \text{for $y$ with } \delta^{\ell +1}\leq\rho (x,y)< \delta^\ell .
\end{split}
\end{equation*}
The condition $\rho(x,y)\geq \delta^{\ell +1}$ implies that $y\notin Q^{\ell +2}(x,t)$ (for any $t$) by Lemma~\ref{lem:existsdyadiccube}, and hence
\begin{equation*}
\begin{split}
\int_{\lbrace y\in X \colon \delta^{\ell+1}\leq\rho (x,y)< \delta^\ell\rbrace} K(x,y)f(y)\, d\sigma (y) & \leq 
\varphi (Q^{\ell -1}(x,t))\int_{\lbrace y\in X\colon \delta^{\ell +1}\leq\rho (x,y)< \delta^\ell\rbrace} f(y)\, d\sigma (y) \\
& \leq 
\varphi (Q^{\ell -1}(x,t))\int_{Q^{\ell -1}(x,t)\setminus Q^{\ell +2}(x,t)} f(y)\, d\sigma (y),
\end{split}
\end{equation*}
where $t$ depends on $x$ and $\ell$. It follows that
\begin{equation*}
\begin{split}
T(f\,d\sigma)(x) &=\sum_{\ell=-\infty}^{\infty}\int_{\lbrace y\in X\colon \delta^{\ell+1}\leq\rho (x,y)< \delta^\ell\rbrace} K(x,y)f(y)\, d\sigma (y)\\
&\leq  \sum_{t=1}^{L}\left(\sum_{k} \varphi (Q^k(x,t))\int_{Q^k(x,t)\setminus Q^{k+3}(x,t)} f\, d\sigma \right) \\
& = \sum_{t=1}^{L} T^{\mathscr{D}^t_{\sigma\omega}}_{3}(f\, d\sigma)(x) \leq 3C\sum_{t=1}^{L} T^{\mathscr{D}^t_{\sigma\omega}}(f\, d\sigma)(x)
\end{split}
\end{equation*}
by Lemma \ref{lem;equivalense;dyad;operators} with $m=3$.
\end{proof}

Proposition~\ref{prop:TheoremC} allows us to reduce the study of potential type operator $T$ to the simpler dyadic models $T^{\mathscr{D}^t_{\sigma\omega}}$. In particular, the proof of our main result, Theorem~\ref{thm:theoremB}, is now completed by the following Proposition:

\begin{proposition}\label{thm:operator;Tdy}
Let $1<p\leq q<\infty$, and let $\sigma$ and $\omega$ be positive Borel-measures on $(X,\rho)$ with the property that $\sigma(B)<\infty$ and $\omega(B)<\infty$ for all balls $B$. Let $T^{\mathscr{D}^t_{\sigma\omega}}$ be a dyadic operator defined in \ref{def:dyadicoperators}. Then
\begin{equation}\label{eq:characterizationinequality}
\| T^{\mathscr{D}^t_{\sigma\omega}} \|_{L^p_\sigma\to L^q_\omega}\approx [\sigma ,\omega]_{S_{p,q}}+ [\omega ,\sigma]_{S_{q',p'}},
\end{equation}
and the constants of equivalence depend only on the geometric structure of $X$, and $p$ and $q$. Here
\[[\sigma ,\omega]_{S_{p,q}}:= \sup_{Q\in\mathscr{D}^t}\sigma(Q)^{-1/p}
\left\| \chi_{Q}T^{\mathscr{D}^t_{\sigma\omega}}(\chi_{Q}\, d\sigma)\right\|_{L^q_\omega} \]
and
\[[\omega ,\sigma]_{S_{q',p'}}:= \sup_{Q\in\mathscr{D}^t}\omega(Q)^{-1/q'}
\left\| \chi_{Q}T^{\mathscr{D}^t_{\sigma\omega}}(\chi_{Q}\, d\omega)\right\|_{L^{p'}_\sigma} \]
are the testing conditions. If $\sigma(Q)=0$ (or $\omega(Q)=0$) for some $Q$ in a testing condition, then $\infty\cdot 0$ is interpreted as $0$.
\end{proposition}
The proof will be given in Section~\ref{sec:proof_for_main_thm}.

\section{Maximum principle for dyadic operators}
In this section, we will prove the so-called maximum principle estimate, which presents an important localization for the dyadic operator $T^{\mathscr{D}^t_{\sigma\omega}}(f\, d\sigma)$. This may be seen as a distinguishing feature of the operator which is the reason why we at this moment study it on its own right. Maximum principle will be utilized in the proof of both the strong type result, Proposition~\ref{thm:operator;Tdy}, in Section~\ref{sec:proof_for_main_thm} as well as the corresponding weak type result in Section~\ref{subsec:weak}. 
\subsection{Maximum principles}
Before stating and proving the maximum principles, we need some preparations. In this section, we will assume that the standard dyadic cubes $Q^k_\alpha\in\mathscr{D}^t_{\sigma\omega}$ are restricted to $k\geq k_0$ with some fixed $k_0\in\Z$; we refer to the maximal cubes $Q^{k_0}$ as the \textit{top-level cubes}. The elementary covering lemma~\ref{lem:maximalcube} is available for such a collection $\mathscr{D}^t_{\sigma\omega}$. 

Let $T^{\mathscr{D}^t_{\sigma\omega}}$ be the dyadic operator associated to such a dyadic system, and let $f\in L^p(X,\sigma)$. We consider the following auxiliary objects:
\begin{equation}\label{eq:levelset}
\begin{split}
\Omega_\rho & :=\{ x\in X\colon T^{\mathscr{D}^t_{\sigma\omega}}(f\, d\sigma)(x)>\rho \},\quad \rho>0,\\ 
\mathscr{Q}_\rho & :=\text{ maximal dyadic cubes $Q\in\mathscr{D}^t_{\sigma\omega}$ such that } \omega(Q\setminus\Omega_\rho)=0.
\end{split}
\end{equation} 

First note that 
\[\Omega_\rho\subseteq \bigcup_{Q\in\mathscr{Q}_\rho} Q   \quad \text{and}\quad 
\omega(\Omega_\rho)=\omega\left(\bigcup_{Q\in\mathscr{Q}_\rho} Q\right)
=\sum_{Q\in \mathscr{\mathscr{Q}_\rho}}\omega(Q),\]
and the union is disjoint.
Indeed, suppose that $x\in \Omega_\rho$ and first assume that $x\in X_{\sigma\omega}$. Then $\{x \}\subseteq \Omega_\rho$. If $x\notin X_{\sigma\omega}$, consider the sequence $\{Q^k(x) \}$ of nested cubes in $\mathscr{D}^t$ which contain $x$. We have
\[\rho<T^{\mathscr{D}^t_{\sigma\omega}}(f\, d\sigma)(x) 
= \sum_{k}\varphi(Q^k (x) )\int_{Q^k(x)\setminus Q^{k+1}(x)}f\,d\sigma . \]
Thus, there exists $i_0\in \Z$ such that
\[\rho< \sum_{k< i_0}\varphi(Q^{k}(x) )\int_{Q^{k}(x)\setminus Q^{k+1}(x)}f\,d\sigma 
\leq T^{\mathscr{D}^t_{\sigma\omega}}(f\, d\sigma)(y)\quad \text{ for }y\in Q^{i_0}(x). \] 
Consequently, $Q^{i_0}(x)\subseteq \Omega_\rho$ and $\Omega_\rho$ is a union of generalized cubes in $\mathscr{D}^t_{\sigma\omega}$. Then consider the (larger) collection of cubes with the property that $\omega(Q\setminus\Omega_\rho)=0$.  Since every $Q\subseteq\Omega_\rho$ is contained in a maximal such cube, the stated $\omega$-a.e. identity of sets follows. In particular, integration over the set $\Omega_\rho$ can be replaced by the sum of integrations over the cubes $Q\in\mathscr{Q}_{\rho}$, and vice versa, when the integration is with respect to the measure $\omega$. We mention that $\Omega_\rho=\bigcup_{Q\in\tilde{\mathscr{Q}}_\rho}Q$ where $\tilde{\mathscr{Q}}_\rho$ is the collection of maximal cubes in $\mathscr{D}^t_{\sigma\omega}$ which are contained in $\Omega_\rho$ -- an observation which was useful in the previous works on the topic. However, since we allow atoms and do not have the non-empty annuli property, the chosen collection $\mathscr{Q}_\rho$ is better suited for our purposes, a fact that transpires in the proof of the first maximum principle below.

\bigskip

We have the following technical variant of the maximum principle studied by Sawyer et al. \cite[formula (3.24)]{SWZ96} (cf. \cite[formula (3.4)]{VW98}):

\begin{lemma}[The first maximum principle]\label{lem:maximumI}   
Suppose $C\geq 2C_K$, where $C_K\geq 1$ is a geometric constant as in the kernel estimates of Lemma~\ref{lemma;varphi}. For $Q\in\mathscr{Q}_{\rho/C}$,
\begin{equation}\label{maximal;principle}
\sup_{Q}T^{\mathscr{D}^t_{\sigma\omega}}(\chi_{Q^c}f\, d\sigma)\leq \rho/2.
\end{equation}
\end{lemma}
\begin{proof}
Fix $\rho>0$ and a geometric constant $C_K\geq 1$ as in Lemma~\ref{lemma;varphi}. Suppose that $C\geq 2C_K$, and let $Q\in \mathscr{Q}_{\rho/C}$. First assume that $Q=\{x\}\notin\mathscr{D}^t, x\in X_{\sigma\omega}$, is a point cube. Suppose, for a contradiction, that
\[T^{\mathscr{D}^t_{\sigma\omega}}(\chi_{\{ x\}^c}f\, d\sigma)(x)=
\sum_{k} \varphi(Q^k(x))\int_{Q^k(x)\setminus Q^{k+1}(x)}f\, d\sigma >\rho/2 .\]
Then there exists $i_0$ such that
\[\sum_{k< i_0} \varphi(Q^k(x))\int_{Q^k(x)\setminus Q^{k+1}(x)}f\, d\sigma > \rho/2\geq \rho/C  \]
showing that $Q^{i_0}(x)\subseteq \mathscr{Q}_{\rho/C}$ and contradicting the maximality of $\{x\}$.

Then assume that $Q\in \mathscr{Q}_{\rho/C}$ is a standard cube, and let $x\in Q$. Given $R\in\mathscr{D}^t$, recall from \ref{def:dyadicoperators} the notation $R^{(1)}$ for the dyadic parent of $R$, and denote by $\hat{R}(x)$ the (unique) dyadic child of $R$ (i.e. a next smaller dyadic cube in $\mathscr{D}^t$ which is contained in $R$) which contains $x$. First note that if $Q^{(1)}$ does not exist (and thus $Q$ is one of the top-level cubes in $\mathscr{D}^t$), then clearly $T^{\mathscr{D}^t_{\sigma\omega}}(\chi_{Q^c}f\, d\sigma)(x)=0$, and \eqref{maximal;principle} follows. Then observe that
\begin{align}\label{eq:splitting}
T^{\mathscr{D}^t_{\sigma\omega}}(\chi_{Q^c}f\, d\sigma)(x) & =\varphi(Q^{(1)})\int_{Q^{(1)}\setminus Q}f\,d\sigma + \sum_{R\in \mathscr{D}^t\atop Q^{(1)}\subsetneq R} \varphi(R)\int_{R\setminus \hat{R}(x)}f\, d\sigma .
\end{align}

Suppose $z\in Q^{(1)}$. We claim that
\begin{align}\label{eq:firstpart}
\varphi(Q^{(1)})\int_{Q^{(1)}\setminus Q}f\,d\sigma &\leq \varphi(Q^{(1)})\int_{Q^{(1)}} f\, d\sigma
=\varphi(Q^{(1)}) \sum_{R\subseteq Q^{(1)}\atop z\in R}\int_{R\setminus \hat{R}(z)} f\, d\sigma + \varphi(Q^{(1)})f(z)\sigma(\{ z\}) \nonumber\\
& \leq C_K\left(\sum_{R\subseteq Q^{(1)}\atop z\in R}\varphi(R)\int_{R\setminus \hat{R}(z)}f\,d\sigma + K(z,z)f(z)\sigma(\{ z\})\right).
\end{align}
Indeed, if for some $R\subseteq Q^{(1)}$ in the sum above, $\{x,y\in B(R)\colon \rho(x,y)\geq cr_{B(R)}\}=\emptyset$ and hence $\varphi(R)=0$, then $R\setminus \hat{R}(z)=\emptyset$ by Lemma~\ref{lemma;varphi}(iii), and the related term vanishes. Thus, in the non-zero terms in the summation we have $\{x,y\in B(R)\colon \rho(x,y)\geq cr_{B(R)}\}\neq\emptyset$, and we may estimate them by Lemma~\ref{lemma;varphi}(i),(ii).

Note that $\hat{R}(x)=\hat{R}(z)$ for $R\supsetneq Q^{(1)}$ and $x\in Q, z\in Q^{(1)}$. Hence, by combining \eqref{eq:splitting} and \eqref{eq:firstpart}, we conclude with
\begin{align*}
T^{\mathscr{D}^t_{\sigma\omega}}(\chi_{Q^c}f\, d\sigma)(x)\leq 
C_K\left(\sum_{Q\in \mathscr{D}^t\atop z\in Q} \varphi(Q)\int_{Q\setminus \hat{Q}(z)}f\, d\sigma +K(z,z)f(z)\sigma(\{ z\})\right) = C_K
T^{\mathscr{D}^t_{\sigma\omega}}(f\, d\sigma)(z)
\end{align*}
where the equality holds for all $z\in X_\sigma^c\cup X_{\omega}$. 
Note that set of points $z\in Q^{(1)}$ where the equality does not hold is $Q^{(1)}\cap (X_\sigma\cap X_\omega^c) = \{z\in Q^{(1)}\colon \sigma(\{ z\})>0 \text{ and } \omega(\{z \})=0\}$ which has an $\omega$-measure zero (since $X_\sigma$ is at most countable). Thus, we have the above estimate valid for $\omega$-a.e. $z\in Q^{(1)}$, and it follows that
\[T^{\mathscr{D}^t_{\sigma\omega}}(\chi_{Q^c}f\, d\sigma)(x)\leq C_K\text{($\omega$-)}\essinf_{Q^{(1)}}T^{\mathscr{D}^t_{\sigma\omega}}(f\,d\sigma)\leq C_K\rho/C\leq \rho/2\]
since the intersection $Q^{(1)}\cap \Omega_{\rho/C}^c$ has a positive $\omega$-measure by the maximality of $Q\in\mathscr{Q}_{\rho/C}$, and by $C\geq 2C_K$.
\end{proof}

The following lemma presents an important localization for $T^{\mathscr{D}^t_{\sigma\omega}}(f\,d\sigma)$:
\begin{lemma}[The second maximum principle]\label{lem:maximumII}
Let $C_m\geq 2C_K$ be a geometric constant as in the first maximum principle \ref{lem:maximumI}. For $Q\in\mathscr{Q}_{\rho/C_m}$,
\begin{equation}\label{maximal;principle2}
T^{\mathscr{D}^t_{\sigma\omega}}(\chi_{Q}f\, d\sigma)(x) > \rho/2 \quad \forall\, x\in Q\cap\Omega_\rho. 
\end{equation}
\end{lemma}
\begin{proof}
This follows immediately from the first maximum principle \ref{lem:maximumI}. Indeed, for $x\in Q\cap\Omega_\rho$,
\begin{equation*}
\begin{split}
T^{\mathscr{D}^t_{\sigma\omega}}(\chi_{Q}f\, d\sigma)(x) & = T^{\mathscr{D}^t_{\sigma\omega}}(f\, d\sigma)(x)-T^{\mathscr{D}^t_{\sigma\omega}}(\chi_{Q^c}f\, d\sigma)(x) >\rho - \rho/2 =\rho/2.
\end{split}
\end{equation*}
\end{proof}

\section{Weak type norm inequality for $T$}\label{subsec:weak}
Let $T$ be a potential type operator and $1<p\leq q<\infty$. The two-weight weak type norm inequality
\begin{equation}\label{est:weaktype_forT}
\|Tf\|_{L^{q,\infty}_\omega}:= \sup_{\rho >0}\rho\, \omega(\{x\in X\colon T(f\,d\sigma)(x)>\rho \})^{1/q}\leq C\| f\|_{L^p_{\sigma}},
\end{equation}
has been studied in a metric space in \cite{VW98, WZ96}. In the Euclidean space with the usual structure, this was treated earlier in \cite{S84} (see also the many references given there). 

In a Sawyer--Wheeden type space $(X,\rho,\mu;\sigma,\omega)$ described in \ref{def:SWspace} with the additional assumption that $X$ has a group structure (in the sense of \cite{SW92}), Verbitsky and Wheeden \cite[Theorem 1.3]{VW98} showed that there is a characterization of \eqref{est:weaktype_forT} by a dual $(Q,Q,Q)$ testing condition which involves testing over all translations of dyadic cubes. We will show that this characterization extends to quasi-metric measure spaces considered in this paper, and that it suffices to test over the cubes $Q\in \bigcup_{t=1}^{L}\mathscr{D}^t$:

\begin{theorem}\label{thm:weakcharacterization}
Let $1<p\leq q<\infty$, and let $\sigma$ and $\omega$ be positive Borel-measures in $(X,\rho)$ with the property that $\sigma(B)<\infty$ and $\omega(B)<\infty$ for all balls $B$. Let $T$ be a potential type operator. Then
\[\|T\|_{L^p_\sigma\to L^{q,\infty}_\omega}\approx [\omega,\sigma]^\ast_{S_{q',p'}}. \]
Here
\[ [\omega,\sigma]^\ast_{S_{q',p'}}:=\sup_{Q}\omega(Q)^{-1/q'}\|\chi_QT^\ast(\chi_Q\, d\omega)\|_{L^{p'}_\sigma}\]
is the dual testing condition where the supremum is over all dyadic cubes $Q\in\bigcup_{t=1}^{L} \mathscr{D}^t$, and $\infty\cdot 0$ is interpreted as $0$.
\end{theorem}

Proposition~\ref{prop:TheoremC} again allows us to reduce to an analogous characterization for the dyadic operators. Thus, the proof of Theorem~\ref{thm:weakcharacterization} is completed by the following lemma.

\begin{lemma}\label{lem:weakdyadic}
Let $T^{\mathscr{D}^t_{\sigma\omega}}$ be a dyadic operator defined in \ref{def:dyadicoperators}. The weak type inequality 
\begin{equation}\label{norm:for_sigma}  
\|T^{\mathscr{D}^t_{\sigma\omega}}(f\,d\sigma)\|_{L^{q,\infty}_\omega}\leq C_1 \|f\|_{L^p_\sigma}
\end{equation}
holds for all $f$, if and only if the dual testing condition 
\begin{equation}\label{test:for_sigma} 
\left(\int_{Q}T^{\mathscr{D}^t_{\sigma\omega}}(\chi_Q\, d\omega)^{p'}\,d\sigma\right)^{1/p'}\leq C_2 \omega(Q)^{1/q'}
\end{equation}
holds for all dyadic cubes $Q\in\mathscr{D}^t$. Moreover, $C_1\approx C_2$.
\end{lemma}

This dyadic result was already shown in \cite[Theorem 3.1]{VW98} but in a Sawyer--Wheeden type space. The proof for Lemma~\ref{lem:weakdyadic} is very similar, and the key step is the maximum principle \eqref{maximal;principle2}. However, the proof requires an approximation argument which in our situation entails some extra work. Thus, it will be necessary to recall most of the argument in \cite{VW98}.

\begin{proof}
To prove that \eqref{norm:for_sigma} implies \eqref{test:for_sigma} with constants $C_1$ and $C_2$ which are equivalent, we may follow the proof given in \cite[Theorem 3.1]{VW98}.

Then assume \eqref{test:for_sigma}. Write $X_\sigma =\{x_k\}_{k\geq 0}$, the enumeration of $\sigma$-atoms (recall that the set $X_\sigma$ is at most countable by the assumption $\sigma(B)<\infty$ for all balls $B$). For a positive integer $n$, consider the measure
\[\sigma_n:= \sigma -\sum_{k=n+1}^{\infty}\sigma(\{x_k\})\delta_{x_k},\] 
which has $n$ atoms. Then $0\leq \sigma_n(E)\leq \sigma(E)$ for all measurable $E\subseteq X$, and \eqref{test:for_sigma} implies the same testing condition with $\sigma$ replaced by $\sigma_n$ on the left-hand side. Assume that $\mathscr{D}^t$ is finite consisting of a finite number of cubes at a fixed top-level together with all the dyadic subcubes of these up to a fixed generation (i.e. the cubes $Q^k_\alpha\in\mathscr{D}^t$ are restricted to $k_0\leq k\leq k_1$, and $z^{k_0}_\alpha\in B(x_0,R)$ with some $x_0\in X$ and a large $R>0$, and all centres $z^{k_0}_\alpha$ of the top-level cubes). Hence, $\mathscr{D}^t_{\sigma_n\omega}$ is finite. 
The maximum principles remain valid for such a collection. 

We will show the desired estimate for the truncated operator $T^{\mathscr{D}^t_{\sigma_n\omega}}$ associated to the family $\mathscr{D}^t_{\sigma_n\omega}$ of finitely many cubes just described. This suffices as long as we provide estimates which are independent of this finite number. 

Consider the level sets $\Omega_\rho$ and the associated collection of cubes $\mathscr{Q}_\rho$ defined in \eqref{eq:levelset}, but with $\mathscr{D}^t_{\sigma\omega}$ replaced by the finite $\mathscr{D}^t_{\sigma_n\omega}$. Let $\rho>0$ and fix a large geometric $C_m\geq 2$ as in the second maximum principle, Lemma~\ref{lem:maximumII}. For $Q\in\mathscr{Q}_{\rho/C_m}$,
\[\rho/2\, \omega(Q\cap\Omega_\rho)\leq \int_{Q}T^{\mathscr{D}^t_{\sigma_n\omega}}(\chi_Qf\,d\sigma_n)\,d\omega.\]
Since 
\[\Omega_\rho=\Omega_\rho\cap \Omega_{\rho/C_m}\subseteq \bigcup_{Q\in\mathscr{Q}_{\rho/C_m}}(Q\cap\Omega_\rho),\]
and the cubes in $\mathscr{Q}_{\rho/C_m}$ are disjoint, we have that
\begin{align*}
\rho/2\,\omega(\Omega_\rho) &
\leq \sum_{Q\in\mathscr{Q}_{\rho/C_m}}\rho/2\, \omega(Q\cap\Omega_\rho)\leq 
\sum_{Q\in\mathscr{Q}_{\rho/C_m}}\int_{Q}T^{\mathscr{D}^t_{\sigma_n\omega}}(\chi_Q f\,d\sigma_n)\,d\omega\\
&=\sum_{Q\in\mathscr{Q}_{\rho/C_m}}\int_{Q}T^{\mathscr{D}^t_{\sigma_n\omega}}(\chi_Q \,d\omega)f\,d\sigma_n \leq 
\sum_{Q\in\mathscr{Q}_{\rho/C_m}}
\left(\int_{Q}T^{\mathscr{D}^t_{\sigma_n\omega}}(\chi_Q \,d\omega)^{p'}\,d\sigma_n\right)^{1/p'}\left(\int_{Q}f^p\,d\sigma_n\right)^{1/p}\\
& \leq C_2\!\!\!\!\sum_{Q\in\mathscr{Q}_{\rho/C_m}}\!\!\omega(Q)^{1/q'}
\left(\int_{Q}f^p\,d\sigma_n\right)^{1/p} 
\leq C_2 \left(\sum_{Q\in\mathscr{Q}_{\rho/C_m}}\omega(Q)^{p'/q'}\right)^{1/p'}
\left(\sum_{Q\in\mathscr{Q}_{\rho/C_m}}\int_{Q}f^p\,d\sigma_n\right)^{1/p}\\
& \leq C_2 \left(\sum_{Q\in\mathscr{Q}_{\rho/C_m}}\omega(Q)\right)^{1/q'}
\left(\sum_{Q\in\mathscr{Q}_{\rho/C_m}}\int_{Q}f^p\,d\sigma_n\right)^{1/p}
\leq C_2\omega(\Omega_{\rho/C_m})^{1/q'}\|f\|_{L^p_\sigma}
\end{align*}
where we used duality, \eqref{test:for_sigma}, $p'\geq q'$ and again the fact that the $Q\in\mathscr{Q}_{\rho/C_m}$ are disjoint and $\omega\big(\bigcup_{Q\in\mathscr{Q}_{\rho/C_m}}Q\big)=\omega(\Omega_{\rho/C_m})$,
and $\sigma_n\leq \sigma$. Hence,
\[\rho^q\omega(\Omega_\rho)=\rho^{q-1}\cdot \rho\omega(\Omega_\rho)
\leq 2C_2C_m^{q-1}\left(\left( \frac{\rho}{C_m}\right)^q\omega(\Omega_{\rho/C_m})\right)^{1/q'}\|f\|_{L^p_\sigma},\]
which yields, for any $N>0$,
\begin{equation}\label{eq:}
\sup_{0<\rho <N}\rho^q\omega(\Omega_\rho)\leq
2C_2 C_m^{q-1}\left(\sup_{0<\rho <N}\rho^q\omega(\Omega_\rho)\right)^{1/q'}  \|f\|_{L^p_\sigma}.
\end{equation}
Note that, since $\Omega_\rho$ is contained in a disjoint union of cubes in $\mathscr{D}^t_{\sigma_n\omega}$ and $\mathscr{D}^t_{\sigma_n\omega}$ is assumed finite,
\[\omega(\Omega_\rho)\leq \sum_{Q\in\mathscr{D}^t_{\sigma_n\omega}}\omega(Q)\leq M<\infty\quad\text{for all } \rho>0 . \]
Hence,
\[\sup_{0<\rho <N}\rho^q\omega(\Omega_\rho) \leq N^qM<\infty \quad\text{for any } 0<N<\infty.\]
Thus, from \eqref{eq:} we obtain
\[\left(\sup_{0<\rho <N}\rho^q\omega(\Omega_\rho)\right)^{1/q}\leq
2C_2 C_m^{q-1} \|f\|_{L^p_\sigma}. \]
By letting $N\to\infty$, this implies
\[\sup_{\rho>0}\rho\, \omega(\Omega_\rho)^{1/q}\leq
2C_2 C_m^{q-1} \|f\|_{L^p_\sigma}\]
which is \eqref{norm:for_sigma} with $\sigma$ replaced by $\sigma_n$ on the left-hand side, and for finite $\mathscr{D}^t$. Since the upper boundary does not depend on $n$ (the number of $\sigma$-atoms) or the number of cubes in $\mathscr{D}^t$, the assertion follows by letting $n\to \infty$ and increasing the number of cubes in $\mathscr{D}^t$. Moreover, we may choose $C_1=2C_2 C_m^{q-1}$.
\end{proof}

\begin{remark}\label{remark:ball}
To rephrase Theorem~\ref{thm:weakcharacterization}, we in particular have that  the dual $(Q,Q,Q)$ testing condition 
\[\|\chi_QT^\ast (\chi_Q\, d\omega)\|_{L^{p'}_\sigma}\leq C \omega(Q)^{1/q'}\quad \forall Q\in \bigcup_{t=1}^{L}\mathscr{D}^t\]
implies the weak type norm inequality 
\[\|T (f\,d\sigma)\|_{L^{q,\infty}_\omega}\leq C\|f\|_{L^{p}_\sigma}. \] 
It follows, by symmetry, that the $(Q,Q,Q)$ testing condition 
\[\|\chi_QT(\chi_Q\, d\sigma)\|_{L^{q}_\omega}\leq C \sigma(Q)^{1/p}\quad \forall Q\in \bigcup_{t=1}^{L}\mathscr{D}^t\]
implies the weak type norm inequality 
\[\|T^\ast (f\,d\omega)\|_{L^{p',\infty}_\sigma}\leq C\|f\|_{L^{q'}_\omega} \]
for the adjoint $T^\ast$ of $T$, and using a simple argument by Verbitsky and Wheeden \cite[p. 3385]{VW98} we see that from this we may deduce
\begin{equation}\label{eq:XQQtest}
\|T(\chi_B\, d\sigma) \|_{L^q_\omega}\leq pC\sigma(B)^{1/p}\quad\text{for all balls } B. 
\end{equation}
This observation improves Lemma~\ref{lem:finite_for_bounded} by giving an upper bound for the norm of $T(\chi_B\,d\sigma)\in L^q_\omega$. Also, \eqref{eq:XQQtest} constitutes the $(X,B,B)$ testing condition. The same argument gives the dual $(X,B,B)$ testing condition
\[\|T^\ast(\chi_B\, d\omega) \|_{L^{p'}_\sigma}\leq q'C\omega(B)^{1/q'}\quad\text{for all balls } B \]
from the dual $(Q,Q,Q)$ testing condition. We mention that there is a characterization of the strong type estimate \eqref{eq:strongtype_S(1)} with $S=T$ by these (less local) $(X,B,B)$ testing conditions which is due to Sawyer, Wheeden and Zhao \cite[Theorem 1.1]{SWZ96} and which may provide a shorter proof for our Theorem~\ref{thm:theoremB}. However, this previous characterization is again in a Sawyer--Wheeden type space described in \ref{def:SWspace}, and can not directly be applied to our setting.
\end{remark}

\section{Strong type norm inequality for dyadic operators}\label{sec:proof_for_main_thm}

Recall that in Section~\ref{sec:dyadic}, we reduced our main result, Theorem~\ref{thm:theoremB}, to Proposition~\ref{thm:operator;Tdy} which is a characterization of norm inequalities by testing conditions for a dyadic operator $T^{\mathscr{D}^t_{\sigma\omega}}$.

Sawyer, Wheeden and Zhao \cite{SWZ96} already showed this sort of dyadic result, but in a Sawyer--Wheeden type space described in \ref{def:SWspace}, and with our function $T^{\mathscr{D}^t_{\sigma\omega}}(f\, d\sigma)$,
\[T^{\mathscr{D}^t_{\sigma\omega}}(f\, d\sigma)(x)=\sum_{k}\varphi(Q^k(x))\int_{Q^k(x)\setminus Q^{k+1}(x)}f(y)\,d\sigma(y)
+\chi_{X_{\sigma\omega}}(x)K(x,x)f(x)\sigma(\{x \}),\]
replaced by the function $T_\mathcal{G}(f\, d\sigma)$,
\[T_\mathcal{G}(f\, d\sigma)(x)=\sum_{k}\varphi(Q^k(x))\int_{Q^k(x)}f(y)\,d\sigma(y) .\]  
Also, their theorem is under the additional technical assumption that $T_\mathcal{G}(\chi_B\, d\sigma)\in L^q(X,\omega)$ for all balls $B\subseteq X$.

The several steps in the proof of Proposition \ref{thm:operator;Tdy} follow closely the ones given in \cite[Theorem 3.2]{SWZ96} for the operator $T_{\mathcal{G}}$ in the mentioned Sawyer--Wheeden setting; in the first steps of the proof, only some technical modifications are needed due to the modified definition of the dyadic operator and the prospective presence of atoms and the lack of the non-empty annuli property. For example, the set corresponding to our set $U_k(Q)$ (defined below) in the original proof, denoted by $U^k_j$, is given by $U^k_j=Q\cap (\Omega_{k+1}\setminus \Omega_{k+2})$ for $Q\in\mathscr{Q}_k$, so that it is a special case of our set with $n=2$. Main differences in comparison to the original proof appear in the final steps of the proof, Lemma~\ref{lem:split_alphabeta} and Lemmata \ref{lem:boundforalpha}-- \ref{lem:final} below, which consist of the main technical aspects of the proof; in this respect, our approach seems a little more articulated than the original one. We will repeat the details of the proof for the reader's convenience even though most of the argument is exactly the same as in \cite{SWZ96}.

\subsection{Proof of Proposition~\ref{thm:operator;Tdy}}
The estimates 
\[\| T^{\mathscr{D}^t_{\sigma\omega}} (\cdot\, d\sigma)\|_{L^p_\sigma\to L^q_\omega}\geq [\sigma ,\omega]_{S_{p,q}} \quad
\text{and} \quad
\|T^{\mathscr{D}^t_{\sigma\omega}} (\cdot\, d\omega)\|_{L^{q'}_\omega\to L^{p'}_\sigma}\geq [\omega ,\sigma]_{S_{q',p'}}\]
are clear, and since 
\[\| T^{\mathscr{D}^t_{\sigma\omega}} (\cdot\, d\omega)\|_{L^{q'}_\omega\to L^{p'}_\sigma}=\| T^{\mathscr{D}^t_{\sigma\omega}} (\cdot\, d\sigma)\|_{L^p_\sigma\to L^q_\omega}\]
by duality, the estimate $\gtrsim$ in the assertion follows. Hence, only the estimate $\lesssim$ requires a proof. So, assume that the testing quantities $[\sigma,\omega]_{S_{p,q}}$ and $[\omega,\sigma]_{S_{q',p'}}$ are finite. We may, without loss of generality, assume that $f\geq 0$ is bounded with bounded support. We may further assume that $\mathscr{D}^t$ consists of dyadic cubes $Q^k_\alpha$ restricted to $ k\geq k_0$ (i.e. that the size of the cubes in $\mathscr{D}^t_{\sigma\omega}$ is bounded from above); we refer to the maximal cubes $Q^{k_0}$ as the \textit{top-level cubes}. The proof of Proposition~\ref{thm:operator;Tdy} will provide an estimate which is independent of $k_0\in \Z$, and the assertion follows for general $\mathscr{D}^t_{\sigma\omega}$ by the Monotone Convergence Theorem.
 
The first step of the proof is Lemma~\ref{lem:first_reduc} below, which requires the following qualitative observation. 

\begin{lemma}\label{lem:finite}
$T^{\mathscr{D}^t_{\sigma\omega}}(\chi_B\, d\sigma)\in L^q_\omega$ for all balls $B$. Consequently, $T^{\mathscr{D}^t_{\sigma\omega}}(f\, d\sigma)\in L^q_\omega$ and thereby, $T^{\mathscr{D}^t_{\sigma\omega}}(f\, d\sigma)<\infty$ $\omega$-a.e. for all bounded $f$ with bounded support. 
\end{lemma}
In the original proof by Sawyer et al. \cite{SWZ96}, Lemma~\ref{lem:finite} is replaced by assuming that $T_\mathcal{G}(\chi_B\, d\sigma)\in L^q_\omega$ for all balls $B$. We mention that Lemma~\ref{lem:finite}, in fact, follows from the weak type result, which we discussed in Section~\ref{subsec:weak}, cf. Remark~\ref{remark:ball}. To preserve self-containedness, we shall, however, provide a direct proof which is independent of the weak type result.

\begin{proof}
Fix a ball $B$. Consider the cubes $Q^{k_0+1}$ which are the dyadic children (i.e. the next smaller cubes) of the top-level cubes. Note that, by the geometric doubling property, there are only finitely many such cubes that intersect $B$, and denote these by $Q_i$. We have
\[\chi_B\leq \sum_{i=1}^{N}\chi_{Q_i}, \]
so that
\[T^{\mathscr{D}^t_{\sigma\omega}}(\chi_B\, d\sigma)\leq 
\sum_{i=1}^{N}T^{\mathscr{D}^t_{\sigma\omega}}(\chi_{Q_{i}}\, d\sigma). \]
It thereby suffices to show that $\|T^{\mathscr{D}^t_{\sigma\omega}}(\chi_{Q_i}\, d\sigma)\|_{L^q_\omega}<\infty$ for all $i$. To this end, fix such $Q_i=:R_0$ and abbreviate $R_1:=R_0^{(1)}$, the dyadic parent of $R_0$ (i.e. $R_1$ is the top-level cube that contains $R_0$). Since
\[\|T^{\mathscr{D}^t_{\sigma\omega}}(\chi_{R_0}\, d\sigma)\|^q_{L^q_\omega}=
 \|\chi_{R_1}T^{\mathscr{D}^t_{\sigma\omega}}(\chi_{R_0}\, d\sigma)\|^q_{L^q_\omega} +
 \|\chi_{R_1^c}T^{\mathscr{D}^t_{\sigma\omega}}(\chi_{R_0}\, d\sigma)\|^q_{L^q_\omega},\]
and the first term on the right is finite by the testing condition $[\sigma,\omega]_{S_{p,q}}<\infty$, it suffices to show that the second term is finite. We will show that, in fact, $\|\chi_{R_1^c}T^{\mathscr{D}^t_{\sigma\omega}}(\chi_{R_0}\, d\sigma)\|^q_{L^q_\omega}=0$. To this end, write 
\begin{align}\label{eq:lemma_finite}
T^{\mathscr{D}^t_{\sigma\omega}}(\chi_{R_0}\, d\sigma)(x) =
\sum_{Q}\chi_Q(x)\varphi(Q^{(1)})\int_{Q^{(1)}\setminus Q}\chi_{R_0}\,d\sigma 
+\chi_{X_{\sigma\omega}}(x)K(x,x)\sigma(\{x\})\chi_{R_0}(x),
\end{align}
where we agree that $Q^{(1)}=\emptyset$ if $Q$ is a top-level cube. Suppose $x\in R^c_1$. The second term in \eqref{eq:lemma_finite} vanishes for such $x$ since $R_0\subseteq R_1$. Note that for any cube $Q$, either $Q\subseteq R_1$ or $Q\cap R_1=\emptyset$ (recall that $R_1$ is a top-level cube). For the cubes $Q\subseteq R_1$, the terms in the sum \eqref{eq:lemma_finite} vanish. Thus, the relevant cubes satisfy 
\begin{equation}\label{eq:summationcontidion}
Q\cap R_1=\emptyset.
\end{equation} 
Moreover, in the non-vanishing terms we must have that $(Q^{(1)}\setminus Q)\cap R_0\neq \emptyset$. Basically, we have two choices: either $Q^{(1)}\subseteq R_0$ or $Q^{(1)}\supsetneq R_0$. First, if $Q^{(1)}\subseteq R_0$, then $Q\subseteq Q^{(1)}\subseteq R_0\subseteq R_1$, so that such $Q$ does not satisfy \eqref{eq:summationcontidion}. Second, for $Q^{(1)}\supsetneq R_0$ we must have that $Q^{(1)}$ is the parent of $R_0$ since $R_0$ is only one level below the top-level. Thus $Q^{(1)}=R_1$, but this implies $Q\subseteq R_1$ so that, again, \eqref{eq:summationcontidion} is not satisfied. Hence, the sum in \eqref{eq:lemma_finite} vanishes, and $T^{\mathscr{D}^t_{\sigma\omega}}(\chi_{R_0}\, d\sigma)=0$ on $R_1^c$.
\end{proof}
 
The proof uses the objects in \eqref{eq:levelset} with $\rho=2^k, k\in \Z$; let us abbreviate
\begin{align*}
\Omega_k & :=\{ x\in X\colon T^{\mathscr{D}^t_{\sigma\omega}}(f\, d\sigma)(x)>2^k \},\quad k\in \Z,\\ 
\mathscr{Q}_k & :=\text{ maximal dyadic cubes $Q\in\mathscr{D}^t_{\sigma\omega}$ such that }\omega(Q\setminus\Omega_k)=0.
\end{align*}
Fix a geometric constant $C_K\geq 1$ as in the kernel estimates of Lemma~\ref{lemma;varphi}, and an integer $n\geq 2$ with the property that $2^{n-1}\geq 2C_K$. Then define
\[U_k(Q) := Q\cap (\Omega_{k+n-1}\setminus \Omega_{k+n}),\quad Q\in\mathscr{Q}_{k}.\]
Note that the sets $U_k(Q)$ are pairwise disjoint in both $k$ and $Q$, and that
\[\Omega_{k+n-1}\setminus \Omega_{k+n} =\bigcup_{Q\in\mathscr{Q}_k}U_k(Q). \]
Also choose $C:=2^{n-1}(\geq 2C_K)$. Then, by the second maximum principle, Lemma~\ref{lem:maximumII}, with $\rho:=C2^k=2^{k+n-1}$, we in particular have that
\begin{equation}\label{maximum;principle3}
T^{\mathscr{D}^t_{\sigma\omega}}(\chi_{Q}f\, d\sigma)(x) > 2^k\quad \forall\, x\in U_k(Q)\subseteq Q\cap\Omega_{k+n-1}. 
\end{equation}
In what follows, we will repeatedly use the positivity of $T^{\mathscr{D}^t_{\sigma\omega}}$, which gives us the pointwise estimate $T^{\mathscr{D}^t_{\sigma\omega}}(f\, d\sigma)\leq T^{\mathscr{D}^t_{\sigma\omega}}(g\, d\sigma)$ for $0\leq f\leq g$.

\begin{lemma}[First reduction]\label{lem:first_reduc}
For a small $\beta>0$ depending only on a geometric constant and $q$, 
\[\| T^{\mathscr{D}^t_{\sigma\omega}}(f\,d\sigma)\|^q_{L^q_\omega}\lesssim \sum_{k}2^{qk}\!\!\!\!\!\!\!\!\!\sum_{Q\in\mathscr{Q}_{k}\atop \omega(U_k(Q))>\beta\omega(Q)}\!\!\!\!\!\!\!\!\!\omega(U_k(Q)) .\]
\end{lemma}

\begin{proof}
With any $\beta\in (0,1)$ we have (recall that $T^{\mathscr{D}^t_{\sigma\omega}}(f\,d\sigma)<\infty$ $\omega$-a.e.)
\begin{align*}
\| T^{\mathscr{D}^t_{\sigma\omega}}(f\,d\sigma)\|^q_{L^q_\omega} & = \sum_{k}\int_{\Omega_{k+n-1}\setminus \Omega_{k+n}}T^{\mathscr{D}^t_{\sigma\omega}}(f\,d\sigma)^q\,d\omega 
\leq \sum_{k}2^{(k+n)q}\omega(\Omega_{k+n-1}\setminus \Omega_{k+n})\\
& = 2^{nq}\sum_{k}2^{qk}\sum_{Q\in\mathscr{Q}_k}\omega(U_k(Q))\\
&\leq 2^{nq}\sum_{k}2^{qk}\hspace{-0.4cm}\sum_{Q\in\mathscr{Q}_k\atop \omega(U_k(Q))\leq \beta \omega(Q)}\hspace{-0.5cm}\omega(U_k(Q)) +2^{nq}\sum_{k}2^{qk}\hspace{-0.4cm}\sum_{Q\in\mathscr{Q}_k\atop \omega(U_k(Q))> \beta \omega(Q)}\hspace{-0.5cm}\omega(U_k(Q))\\
&=:\Sigma_1+\Sigma_2.
\end{align*}

Observe that
\[\sum_{Q\in\mathscr{Q}_k}\omega(Q) =
\omega(\Omega_k)
=\sum_{j\geq k}\omega\big(\Omega_j\setminus\Omega_{j+1}\big). \]
We estimate
\begin{align*}
\Sigma_1 &\leq 2^{nq}\beta\sum_{k}2^{qk}\sum_{Q\in\mathscr{Q}_k}\omega(Q) = 2^{nq}\beta \sum_{k}\left(2^{qk}\sum_{j\geq k}\omega\big(\Omega_j\setminus\Omega_{j+1}\big) \right)\\
& =2^{nq}\beta \sum_{j}\left(\omega\big( \Omega_j\setminus\Omega_{j+1}\big) \sum_{k\leq j}2^{qk}\right)\\
& = \frac{2^{nq}\beta}{1-2^{-q}}\sum_{j}2^{qj}\omega(\Omega_{j}\setminus\Omega_{j+1})
\leq \frac{2^{nq}\beta}{1-2^{-q}}\|T^{\mathscr{D}^t_{\sigma\omega}}(f\,d\sigma) \|^q_{L^q_\omega}.
\end{align*}
Here $\|T^{\mathscr{D}^t_{\sigma\omega}}(f\,d\sigma) \|_{L^q_\omega}$ is finite by Lemma~\ref{lem:finite} and thus, subtractable. Then choose $\beta\in (0,1)$ so small that $2^{nq}\beta/(1-2^{-q})<1/2$ to complete the proof.
\end{proof}

\begin{lemma}[Second reduction]
\begin{align*}
\| T^{\mathscr{D}^t_{\sigma\omega}}(f\,d\sigma)\|^q_{L^q_\omega} & \lesssim \sum_{k}\sum_{Q\in \mathscr{Q}_k}
\frac{\omega(U_k(Q))}{\omega(Q)^q}\left(\int_{Q} fT^{\mathscr{D}^t_{\sigma\omega}}(\chi_{U_k(Q)}\,d\omega)d\sigma \right)^q \\
& =: \sum_{k}\sum_{Q\in \mathscr{Q}_k}
\frac{\omega(U_k(Q))}{\omega(Q)^q}
\left(\theta_k(Q) + \gamma_k(Q)\right)^q,
\end{align*}
where 
\[\theta_k(Q):=\int_{Q\setminus \Omega_{k+n}}fT^{\mathscr{D}^t_{\sigma\omega}}(\chi_{U_k(Q)}\,d\omega)\,d\sigma \quad\text{and}\quad  \gamma_k(Q):=\int_{Q\cap \Omega_{k+n}}fT^{\mathscr{D}^t_{\sigma\omega}}(\chi_{U_k(Q)}\,d\omega)\,d\sigma .\]
\end{lemma}
\begin{proof}
Suppose $Q\in \mathscr{Q}_k$ and $\omega(U_k(Q))>\beta \omega(Q)>0$. By the maximum principle \eqref{maximum;principle3} and duality, 
\begin{align*}
2^k & \leq \frac{1}{\omega(U_k(Q))}\int_{U_k(Q)} T^{\mathscr{D}^t_{\sigma\omega}}(\chi_{Q}f\, d\sigma) d\omega 
= \frac{1}{\omega(U_k(Q))}\int_{Q} T^{\mathscr{D}^t_{\sigma\omega}}(\chi_{U_k(Q)}\,d \omega) fd\sigma \\
& \lesssim \frac{1}{\omega(Q)}\int_{Q} T^{\mathscr{D}^t_{\sigma\omega}}(\chi_{U_k(Q)}\, d\omega) fd\sigma .
\end{align*}
Then use the first reduction; once the summation condition is used as above, it can be dropped, and the result only increases. 
\end{proof}

\begin{lemma}[Bound for $\theta_k(Q)$]
\[\sum_{k}\sum_{Q\in\mathscr{Q}_k}
\frac{\omega(U_k(Q))}{\omega(Q)^q}\left(
\theta_k(Q) \right)^q
\lesssim [\omega ,\sigma]_{S_{q',p'}}^q\, \| f \|_{L^p_\sigma}^q \]
\end{lemma}
\begin{proof}
We estimate by H\"older's inequality,
\begin{align*}
\theta_k(Q) &
\leq \left(\int_{Q\setminus \Omega_{k+n}} \big(T^{\mathscr{D}^t_{\sigma\omega}}(\chi_{U_k(Q)}\,d\omega)\big)^{p'}d\sigma\right)^{1/p'}
\left(\int_{Q\setminus \Omega_{k+n}}f^p\, d\sigma\right)^{1/p}\\
& \leq \|\chi_QT^{\mathscr{D}^t_{\sigma\omega}}(\chi_{Q}\,d\omega)  \|_{L^{p'}_\sigma}
\left(\int_{Q\setminus \Omega_{k+n}}f^p\, d\sigma\right)^{1/p}\quad\text{since $U_k(Q)\subseteq Q$}\\
& \leq \omega(Q)^{1/q'}[\omega ,\sigma]_{S_{q',p'}}\left(\int_{Q\setminus \Omega_{k+n}}f^p\, d\sigma\right)^{1/p}.
\end{align*}
Hence (using $q-q/q'=1$ and $\omega(U_k(Q))\leq \omega(Q)$),
\begin{align*}
\frac{\omega(U_k(Q))}{\omega(Q)^q}\left(
\theta_k(Q) \right)^q &  \leq [\omega,\sigma]_{S_{q',p'}}^q \left(\int_{Q\setminus \Omega_{k+n}}f^p\, d\sigma\right)^{q/p}.
\end{align*}
Finally, using $p\leq q$, we obtain
\begin{align*}
\sum_{k}\sum_{Q\in\mathscr{Q}_k} &
\frac{\omega(U_k(Q))}{\omega(Q)^q}\left(
\theta_k(Q)\right)^q  \leq
 [\omega,\sigma]_{S_{q',p'}}^q\sum_{k}\sum_{Q\in\mathscr{Q}_k}  \left(\int_{Q\setminus \Omega_{k+n}}f^p\, d\sigma\right)^{q/p}\\
 & \leq  [\omega ,\sigma]_{S_{q',p'}}^q\left(\sum_{k}\sum_{Q\in\mathscr{Q}_k}  \int_{Q\setminus \Omega_{k+n}}f^p\, d\sigma\right)^{q/p}
   =  [\omega,\sigma]_{S_{q',p'}}^q\left(\sum_{k} \int_{\Omega_k\setminus \Omega_{k+n}}f^p\, d\sigma\right)^{q/p}\\
   & \leq [\omega,\sigma]_{S_{q',p'}}^q\big(n\| f\|_{L^p_\sigma}^p\big)^{q/p}
   = n^{q/p}[\omega,\sigma]_{S_{q',p'}}^q \| f\|_{L^p_\sigma}^q
\end{align*}
since 
\[\sum_{k}\chi_{\Omega_k\setminus \Omega_{k+n}}(x)\leq n\quad\text{for all }x\in X. \]
\end{proof}

\subsection{The main technicalities of the proof}
The analysis of $\gamma_k(Q)$ (the integration over $Q\cap\Omega_{k+n}$) consists of the main technical aspects of the proof. We begin with the following lemma.

\begin{lemma}\label{lem:local_two} 
The function $T^{\mathscr{D}^t_{\sigma\omega}}(\chi_{U_k(Q)}\,d\omega)$, $Q\in\mathscr{Q}_k$, is constant on each $R\in\mathscr{Q}_{k+n}$. 
\end{lemma}
\begin{proof}
We observe that $\chi_{U_k(Q)}=0$ on each $R\in\mathscr{Q}_{k+n}$ since $U_k(Q)\cap R=\emptyset$ for all $R\in \mathscr{Q}_{k+n}$ due to the fact that $\Omega_{k+n}$ was removed when defining $U_k(Q)$. Also observe that $\hat{S}(y)=\hat{S}(x)$ if $S\supsetneq R$ and $x,y\in R$ (recall the notation $\hat{S}(x)$ for the next smaller dyadic cube in $\mathscr{D}^t$ which is contained in $S$ and contains $x$). Hence, for $x,y\in R$, $R\in \mathscr{Q}_{k+n}$, we have
\begin{align*}
T^{\mathscr{D}^t_{\sigma\omega}}(\chi_{U_k(Q)}\,d\omega)(x) & =
\sum_{S\in\mathscr{D}^t\atop R\subsetneq S}\varphi(S)\int_{S\setminus \hat{S}(x)}\chi_{U_k(Q)}d\omega \\
& =\sum_{S\in\mathscr{D}^t\atop R\subsetneq S}\varphi(S)\int_{S\setminus \hat{S}(y)}\chi_{U_k(Q)}d\omega = T^{\mathscr{D}^t_{\sigma\omega}}(\chi_{U_k(Q)}\,d\omega)(y).
\end{align*}
The claimed constancy follows. 
\end{proof}

\begin{lemma}[Analysis of $\gamma_k(Q)$]
For $Q\in\mathscr{Q}_k$,
\begin{equation}\label{eq:firstboundforgamma}
\gamma_k(Q):=
\int_{Q\cap \Omega_{k+n}} fT^{\mathscr{D}^t_{\sigma\omega}}(\chi_{U_k(Q)}\,d\omega)\, d\sigma 
= \sum_{R\in\mathscr{Q}_{k+n}\atop R\subseteq Q}
\left(\langle f \rangle^{\sigma}_{R} \int_{R} T^{\mathscr{D}^t_{\sigma\omega}}(\chi_{Q}\, d\omega) \, d\sigma \right). 
\end{equation}
\end{lemma}
For the integral average of $f$, we have introduced the short hand notation
\[\langle f \rangle_R^{\sigma}:=\frac{1}{\sigma(R)}\int_{R}f\,d\sigma .\]

\begin{proof}
We note that there is the identity of sets 
\[Q\cap\Omega_{k+n} =\bigcup\{R\in \mathscr{Q}_{k+n}\colon R\subseteq Q  \}. \]
Indeed, each $R\in \mathscr{Q}_{k+n}$ is contained in $\Omega_{k+n}\subseteq \Omega_k\subseteq \bigcup \{ Q\colon Q\in\mathscr{Q}_k\}$ (disjoint), and those which intersect $Q$ must be contained in it. 

Since $T^{\mathscr{D}^t_{\sigma\omega}}(\chi_{U_k(Q)}\,d\omega)$, $Q\in\mathscr{Q}_k$, is constant, say with value $c_k(Q,R)$, on every $R\in\mathscr{Q}_{k+n}$, we obtain 
\begin{align*}
\gamma_k(Q) & =\sum_{R\in \mathscr{Q}_{k+n}\atop R\subseteq Q }\int_{R}fT^{\mathscr{D}^t_{\sigma\omega}}(\chi_{U_k(Q)}\,d\omega)\, d\sigma
= \sum_{R\in \mathscr{Q}_{k+n}\atop R\subseteq Q}c_k(Q,R)\int_{R}f\, d\sigma\\
& =\sum_{R\in \mathscr{Q}_{k+n}\atop R\subseteq Q}\frac{1}{\sigma(R)}\int_{R}f\, d\sigma \int_{R}T^{\mathscr{D}^t_{\sigma\omega}}(\chi_{U_k(Q)}\,d\omega)\, d\sigma \leq 
\sum_{R\in \mathscr{Q}_{k+n}\atop R\subseteq Q}\langle f \rangle_R^{\sigma} \int_{R}T^{\mathscr{D}^t_{\sigma\omega}}(\chi_{Q}\,d\omega)\, d\sigma .
\end{align*}
\end{proof}

\begin{lemma}[Further analysis of $\gamma_k(Q)$]\label{lem:split_alphabeta}
For $Q\in\mathscr{Q}_k$ and any $P\supseteq Q$,
\begin{align*}
\gamma_k(Q)&\leq 
4\langle f \rangle_{P}^\sigma\int_{U_k(Q)} T^{\mathscr{D}^t_{\sigma\omega}}(\chi_{P}\, d\sigma) \,  d\omega + 
\sum_{R\in\mathscr{Q}_{k+n}, R\subseteq Q \atop \langle f \rangle_{R}^\sigma>4\langle f \rangle_{P}^\sigma}
\langle f \rangle_{R}^\sigma \int_{R} T^{\mathscr{D}^t_{\sigma\omega}}(\chi_{Q}\, d\omega) \,d\sigma \\
& =: \alpha_k(Q,P) + \beta_k(Q,P). 
\end{align*}
\end{lemma}
\begin{proof}
We split the summation in \eqref{eq:firstboundforgamma} over $\{R\in\mathscr{Q}_{k+n}\colon R\subseteq Q\}$ into two according to whether $\langle f \rangle_{R}^{\sigma}\leq 4\langle f \rangle_{P}^{\sigma}$ or not. In the first subseries, we use the fact that the cubes $R\in\mathscr{Q}_{k+n}$ are disjoint and contained in $Q\subseteq P$, so that 
\[ \sum_{R}\int_{R} T^{\mathscr{D}^t_{\sigma\omega}}(\chi_{U_k(Q)}\, d\omega)\, d\sigma  \leq 
 \int_{P} T^{\mathscr{D}^t_{\sigma\omega}}(\chi_{U_k(Q)}\, d\omega) \, d\sigma  \leq 
  \int_{U_k(Q)} T^{\mathscr{D}^t_{\sigma\omega}}(\chi_{P}\, d\sigma) \,  d\omega \]
where we used duality in the final step. 
\end{proof}

\subsection{Principal cubes}\label{principalcubes}
We shall estimate the sum over the terms $\alpha_k(Q,P)$ and $\beta_k(Q,P)$ separately, and for every $Q$, we choose a particular $P=\Pi(Q)\supseteq Q$ which is defined by introducing the so-called \textit{principal cubes} (cf. \cite[p. 804]{MW77}). These, in turn, are defined by a stopping time argument as follows:

\begin{definition}\label{def:principalcubes}
Let $\mathscr{P}_0$ consist of all the maximal (hence disjoint) dyadic cubes (recall that the size of the cubes in $\mathscr{D}^t$ is assumed to be bounded from above), and inductively, if $\mathscr{P}_j$ has been defined, let
\[\mathscr{P}_{j+1}:= \bigcup_{P\in \mathscr{P}_j}\left\{ R\subsetneq P\colon 
\text{$R$ is maximal subcube with the property that }
\langle f\rangle_R^\sigma > 2 \langle f\rangle_P^\sigma \right\}. \]
Further define $\mathscr{P}:=\bigcup_{k=0}^{\infty}\mathscr{P}_k$, the family of principal cubes, and for each $Q\in\mathscr{D}^t$, denote
\[\Pi(Q):= \text{ the minimal $P\in \mathscr{P}$ which contains $Q$ }. \]
\end{definition}

\begin{remark}\label{rem:principal}
By definition, the principal cubes have the following properties:
\begin{equation*}\label{lem:principalcubes}
\begin{split}
(i) & \quad \text{If } P_1,P_2\in\mathscr{P} \text{ and } P_1\subsetneq P_2, \text{ then } \langle f\rangle_{P_1}^\sigma > 2 \langle f\rangle_{P_2}^\sigma; \\
(ii) & \quad \langle f\rangle_Q^\sigma \leq 2 \langle f\rangle_{\Pi(Q)}^\sigma.
\end{split}
\end{equation*}
\end{remark}

The collection $\mathscr{P}$ has the useful property that the sum $\sum_{P\in\mathscr{P}}(\langle f\rangle^\sigma_P)^p\chi_P$ is controlled pointwise by a certain dyadic maximal function of $f$. In fact, this property is enjoyed by any collection of dyadic cubes wherein there is no repetition (in the sense described in Lemma~\ref{lem:mainlemma} below), which has the property (i) of Remark~\ref{rem:principal}:

\begin{lemma}\label{lem:mainlemma}
Suppose $\mathscr{R}\subseteq\mathscr{D}^t$ is a collection of dyadic cubes with the properties that for $R_i=R^{k_i}_{\alpha_i}\in\mathscr{R}, i=1,2$, the equality $R_1=R_2$ implies $(k_1,\alpha_1) =(k_2,\alpha_2)$, and $R_1\subsetneq R_2$ implies $\langle f\rangle^\sigma_{R_1}>2\langle f\rangle^\sigma_{R_2}$. Then for all $x\in X$ and $1<p<\infty$,
\[\sum_{R\in \mathscr{R}}\big(\langle f\rangle _R^\sigma\big)^p\chi_R(x) \leq 2 \big(M_\sigma f(x)\big)^p. \]
\end{lemma}
The notation $M_\sigma$ stands for the \textit{(weighted) dyadic Hardy--Littlewood maximal operator} given by
\[M_\sigma f(x):=\sup_{Q\in\mathscr{D}^t}\frac{\chi_Q(x)}{\sigma(Q)}\int_{Q}\abs{f}d\sigma
= \sup_{Q\in\mathscr{D}^t}\chi_Q(x)\langle f\rangle _Q^\sigma \quad \text{for }f\in L^1_{\loc}(X,\sigma) \text{ and } x\in X. \]
Note that Lemma~\ref{lem:mainlemma}, in particular, applies to the collection $\mathscr{P}$ of principal cubes.

\begin{proof}[Proof of Lemma~\ref{lem:mainlemma}]
Fix $x\in X$, and recall that the size of cubes in $\mathscr{D}^t$ is assumed bounded. Denote by $R_0=R_0(x)$ the largest (top-level) cube in $\mathscr{R}$ that contains $x$, and further, by $R_k=R_k(x), k\geq 0,$ the (decreasing) sequence of cubes in $\mathscr{R}$ that contain $x$ with $R_{k+1}\subsetneq R_k$. Then, for non-negative integers $N\geq 0$ and $k< N$, $\langle f\rangle _{R_{N}}^\sigma >2\langle f\rangle _{R_{N-1}}^\sigma>\ldots >2^{N-k}\langle f\rangle _{R_{k}}^\sigma$ so that $\langle f\rangle _{R_{k}}^\sigma<2^{k-N}\langle f\rangle _{R_{N}}^\sigma$. Thus, for a truncated sum we have
\begin{align*}
\sum_{k=0}^{N}\big(\langle f\rangle _{R_k}^\sigma\big)^{p}\leq
\big(\langle f\rangle _{R_N}^\sigma\big)^{p}
\left(\sum_{k=0}^{N}2^{(k-N)p}\right)
\leq \big(\langle f\rangle _{R_N}^\sigma\big)^{p} \sum_{k=0}^{\infty}2^{-kp}
\leq 2 \big(M_\sigma f(x)\big)^p
\end{align*}
by the definition of $M_\sigma f(x)$. The assertion follows by letting $N\to \infty$.
\end{proof}

We recall the following well-known result which we refer to as \textit{the universal maximal function estimate} and which will be of use later in the proof: For any measure $w$ and $1<p<\infty$, 
\begin{equation}\label{eq:universalmaximal}
\|M_w f \|_{L^p_w} \leq p'\|f \|_{L^p_w}.
\end{equation}

\begin{lemma}[Bound for $\alpha_k(Q,P)$]\label{lem:boundforalpha}
\[\sum_{k}\sum_{Q\in\mathscr{Q}_k}\frac{\omega(U_k(Q))}{\omega(Q)^q}\big(\alpha_k(Q,\Pi(Q))\big)^q\lesssim [\sigma ,\omega]_{S_{p,q}}^q\|f \|_{L^p_\sigma}^q .\]
\end{lemma}
\begin{proof}
We re-organize the sum as
\[\sum_{k}\sum_{Q\in\mathscr{Q}_k} =
\sum_{P\in \mathscr{P}}\sum_{k}\sum_{Q\in\mathscr{Q}_k \atop \Pi(Q)=P}. \]
For a moment, fix $P\in \mathscr{P}$. First note that by H\"older's inequality and since $q/q'=q-1$,
\begin{align*}
\left(\int_{U_k(Q)} T^{\mathscr{D}^t_{\sigma\omega}}(\chi_{\Pi(Q)}\,d\sigma) \,  d\omega\right)^q & \leq 
\omega(U_k(Q))^{q-1}\int_{U_k(Q)} \big(T^{\mathscr{D}^t_{\sigma\omega}}(\chi_{\Pi(Q)}\, d\sigma)\big)^q \, d\omega .
\end{align*}
Thus, 
\begin{align*}
\sum_{k}\sum_{Q\in\mathscr{Q}_k \atop \Pi(Q)=P}\frac{\omega(U_k(Q))}{\omega(Q)^q} \big(\alpha_k(Q, \Pi(Q))\big)^q & \leq 
(4\langle f \rangle_{P}^{\sigma})^q\sum_{k}\sum_{Q\in\mathscr{Q}_k \atop \Pi(Q)=P}\frac{\omega(U_k(Q))^q}{\omega(Q)^q} \int_{U_k(Q)} \big(T^{\mathscr{D}^t_{\sigma\omega}}(\chi_{P}\, d\sigma)\big)^q \, d\omega \\
& \leq (4\langle f \rangle_{P}^{\sigma})^q
\int_{P} \big(T^{\mathscr{D}^t_{\sigma\omega}}(\chi_{P}\, d\sigma)\big)^q \, d\omega\\
&\leq (4\langle f \rangle_{P}^{\sigma})^q \sigma(P)^{q/p} [\sigma ,\omega]^q_{S_{p,q}}
\end{align*}
since the sets $U_k(Q)\subseteq Q\subseteq \Pi(Q)=P$ are pairwise disjoint in both $k$ and $Q$. 
Hence,
\begin{align*}
\sum_{k}\sum_{Q\in\mathscr{Q}_k}\frac{\omega(U_k(Q))}{\omega(Q)^q}\big(\alpha_k(Q, \Pi(Q))\big)^q
\lesssim [\sigma ,\omega]^q_{S_{p,q}} 
\sum_{P\in \mathscr{P}}\big( (\langle f \rangle_{P}^{\sigma})^p \sigma(P)\big)^{q/p} .
\end{align*}
Here, since $q\geq p$,
\begin{align*}
\sum_{P\in \mathscr{P}}\big( (\langle f \rangle_{P}^{\sigma})^p\sigma(P)\big)^{q/p} &
\leq \left( \sum_{P\in \mathscr{P}}(\langle f \rangle_{P}^{\sigma})^p\sigma(P)\right)^{q/p} = 
\left(\int_{X} \sum_{P\in \mathscr{P}}\chi_P(\langle f \rangle_{P}^{\sigma})^p\, d\sigma  \right)^{q/p} \\
& \lesssim \left(\int_{X} \big(M_\sigma f(x) \big)^p\, d\sigma \right)^{q/p}
\lesssim \| f\|_{L^p_\sigma}^q 
\end{align*}
where we used Lemma~\ref{lem:mainlemma} in the second-to-last estimate, and the universal maximal function  estimate \eqref{eq:universalmaximal} in the last estimate.
\end{proof}

\begin{lemma}[Bound for $\beta_k(Q,P)$, I]\label{lem:boundforbetaI}
For $Q\in\mathscr{Q}_k$ and any $P\supseteq Q$,
\[\beta_k(Q,P)\leq
\omega(Q)^{1/q'}[\omega ,\sigma]_{S^{q',p'}} 
\left( \sum_{R\in\mathscr{Q}_{k+n}, R\subseteq Q \atop \langle f \rangle_{R}^\sigma>4\langle f \rangle_{P}^\sigma} \!\!\!
 \big( \langle f \rangle_{R}^\sigma \big)^p \sigma(R) \right)^{1/p}.\]
\end{lemma}
\begin{proof}
We apply H\"older's inequality, first with respect to integration, then with respect to summation:
\begin{align*}
\sum_{R}\langle f \rangle_{R}^\sigma \int_{R}T^{\mathscr{D}^t_{\sigma\omega}}(\chi_Q \, d\omega)\, d\sigma  & \leq 
\sum_{R}\langle f \rangle_{R}^\sigma \, \sigma(R)^{1/p} 
\left( \int_{R}\big(T^{\mathscr{D}^t_{\sigma\omega}}(\chi_Q \,d\omega)\big)^{p'} d\sigma \right)^{1/p'}\\
& \leq \left( \sum_{R}\big(\langle f \rangle_{R}^\sigma\big)^p \sigma(R) \right)^{1/p} \left(  \sum_{R} \int_{R}\big(T^{\mathscr{D}^t_{\sigma\omega}}(\chi_Q \,d\omega)\big)^{p'} d\sigma  \right)^{1/p'}.
\end{align*}
Here the summation condition is the same as in the assertion, and we may estimate the second factor by
\[\left(  \sum_{R} \int_{R}\big(T^{\mathscr{D}^t_{\sigma\omega}}(\chi_Q\, d\omega)\big)^{p'} d\sigma  \right)^{1/p'} \leq 
\left( \int_{Q}\big(T^{\mathscr{D}^t_{\sigma\omega}}(\chi_Q \,d\omega)\big)^{p'} d\sigma \right)^{1/p'} \leq
\omega(Q)^{1/q'}[\omega ,\sigma]_{S_{q',p'}} \]
since the relevant $R$ are disjoint and contained in $Q$.
\end{proof}
Note that from now on, the operator no longer appears in the estimates, and the remaining analysis only amounts to estimating the integral averages $\langle f \rangle_{R}^\sigma$. 

The following lemma illustrates the advantage of the chosen summation condition.

\begin{lemma}\label{lem:sparseness}
Let $k_1\equiv k_2$ (mod $n$), and suppose that
\[Q_i\in \mathscr{Q}_{k_i}, \quad R_{i}\in \mathscr{Q}_{k_i+n},
\quad R_{i}\subseteq Q_i \quad\text{and}\quad  
\langle f \rangle_{R_i}^\sigma>4\langle f \rangle_{\Pi(Q_i)}^\sigma,\quad i=1,2 .\]
Then, $R_1=R_2$ implies $(k_1,Q_1)=(k_2,Q_2)$, and $R_1\subsetneq R_2$ implies $\langle f \rangle_{R_1}^\sigma>2\langle f \rangle_{R_2}^\sigma$.
\end{lemma}
\begin{proof}
First suppose that $R_1 = R_2$. Assume, for a contradiction, that $k_1\neq k_2$. Without loss of generality, assume $k_1>k_2$, and thus $k_1\geq k_2+n$. This implies $\Omega_{k_1}\subseteq \Omega_{k_2+n}$. Since $Q_1\in\mathscr{Q}_{k_1}$ is contained in some (unique) $R\in\mathscr{Q}_{k_2+n}$, and $Q_1$ contains $R_1=R_2\in \mathscr{Q}_{k_2+n}$, we have $Q_1=R_1$. Hence, $\langle f \rangle_{R_1}^\sigma \leq 2\langle f \rangle_{\Pi(Q_1)}^\sigma$ by property (ii) of Remark~\ref{rem:principal}, a contradiction. Thus, we must have $k_1=k_2$, and thereby also $Q_1=Q_2$ since both contain $R_1$, and different elements of $\mathscr{Q}_{k_1}$ are disjoint.

Then suppose that $R_1\subsetneq R_2$. Then $k_1>k_2$, and thus $k_1\geq k_2+n$. Since $Q_1\in\mathscr{Q}_{k_1}$ is again contained in some $R\in\mathscr{Q}_{k_2+n}$, and $Q_1$ and $R_2\in \mathscr{Q}_{k_2+n}$ intersect on $R_1$, we have $Q_1\subseteq R_2$, and thereby $\Pi(Q_1)\subseteq \Pi(R_2)$. It follows that
\[\langle f \rangle_{R_1}^\sigma > 4\langle f \rangle_{\Pi(Q_1)}^\sigma 
\geq 4\langle f \rangle_{\Pi(R_2)}^\sigma \geq 2\langle f \rangle_{R_2}^\sigma \]
where we used the assumption and Remark~\ref{rem:principal}.
\end{proof}

\begin{lemma}[Bound for $\beta_k(Q,P)$, II]\label{lem:final}
\[\sum_{k}\sum_{Q\in\mathscr{Q}_k}\frac{\omega(U_k(Q))}{\omega(Q)^q}\big(\beta_k(Q,\Pi(Q))\big)^q\lesssim [\omega ,\sigma]_{S_{q',p'}}^q\|f \|_{L^p_\sigma}^q .\]
\end{lemma}
\begin{proof}
By Lemma~\ref{lem:boundforbetaI} with $P=\Pi(Q)$,
\begin{align*}
\sum_{k}\sum_{Q\in\mathscr{Q}_k}\frac{\omega(U_k(Q))}{\omega(Q)^q}\big(\beta_k(Q)\big)^q & \lesssim  [\omega ,\sigma]_{S_{q',p'}}^q
\sum_{k}\sum_{Q\in\mathscr{Q}_k}\left( \sum_{R\in\mathscr{Q}_{k+n}, R\subseteq Q \atop \langle f \rangle_{R}^\sigma>4\langle f \rangle_{\Pi(Q)}^\sigma} 
 \big( \langle f \rangle_{R}^\sigma \big)^p \sigma(R) \right)^{q/p}\\
 & \leq  [\omega ,\sigma]_{S_{q',p'}}^q
\left(\sum_{k}\sum_{Q\in\mathscr{Q}_k} \sum_{R\in\mathscr{Q}_{k+n}, R\subseteq Q \atop \langle f \rangle_{R}^\sigma>4\langle f \rangle_{\Pi(Q)}^\sigma} 
 \big( \langle f \rangle_{R}^\sigma \big)^p \sigma(R) \right)^{q/p}
 \end{align*}
since $q-q/q'=1$ and $U_k(Q)\subseteq Q$, and $q\geq p$. We split the sum into $n$ according to the condition $k\equiv \ell$ (mod $n$), $\ell=1,2,\ldots ,n$, and consider one of these subsums. Denote by $\mathscr{R}$ the collection of all $R$ that appear in such subsum. By Lemma~\ref{lem:sparseness}, any given $R$ appears at most once (i.e. is associated to at most one pair $(k,Q)$), and for two different $R_1\subsetneq R_2$ we have $\langle f \rangle_{R_1}^\sigma>2\langle f \rangle_{R_2}^\sigma$. Thus, Lemma~\ref{lem:mainlemma} is available, and the proof is completed by
\begin{align*}
\sum_{k}\sum_{Q\in\mathscr{Q}_k} \sum_{R\in\mathscr{Q}_{k+n}, R\subseteq Q \atop \langle f \rangle_{R}^\sigma>4\langle f \rangle_{\Pi(Q)}^\sigma} 
 \big( \langle f \rangle_{R}^\sigma \big)^p \sigma(R) & =
 \sum_{R\in\mathscr{R}} \big( \langle f \rangle_{R}^\sigma \big)^p \sigma(R) = 
 \int_{X}\sum_{R\in\mathscr{R}} \chi_{R}\big( \langle f \rangle_{R}^\sigma \big)^p\, d\sigma \\
 & \lesssim \int_{X} \big( M_\sigma f \big)^pd\sigma  \lesssim
 \| f\|_{L^p_\sigma}^p 
\end{align*}
where we used the universal maximal function  estimate \eqref{eq:universalmaximal} in the last estimate.
\end{proof}


\section{A characterization of norm estimates for the maximal operator}\label{sec:MO}

In this section we derive a characterization of the two-weight norm inequality
\begin{equation}\label{est:strong_for_M}
\left(\int_{X} (M_{\mu,\gamma}f)^q\,d\omega\right)^{1/q}\leq C\left( \int_{X}f^p\,d\sigma \right)^{1/p},\quad f\in L^p_\sigma ,
\end{equation}
for the \textit{fractional maximal operator} $M_{\mu,\gamma}$ defined by
\begin{equation}\label{def:operatorM}
M_{\mu,\gamma}f(x):=\sup_{B}\frac{\chi_{B}(x)}{\mu(B)^{1-\gamma}}\int_{B}\abs{f}\, d\mu, \quad x\in X, \quad 0\leq \gamma <1.
\end{equation}
Our characterization is in a space of homogeneous type $(X,\rho,\mu)$, and the integration inside the operator is with respect to an underlying doubling measure $\mu$ which satisfies the doubling condition \eqref{def:doubling}. The positive Borel-measures $\sigma$ and $\omega$ appearing in the norm estimate \eqref{est:strong_for_M} are not assumed to satisfy the doubling condition. Note that with $\gamma=0$, \eqref{def:operatorM} gives the classical Hardy--Littlewood maximal function.

The characterization of norm estimates \eqref{est:strong_for_M} in Euclidean spaces were first obtained by Sawyer \cite{S82} where it was shown that \eqref{est:strong_for_M} is characterized by a $(Q,Q,Q)$ testing condition where $Q$ denotes an arbitrary cube in $\R^n$. A new and simpler proof of Sawyer's result was given by D. Cruz-Uribe \cite{CruzUribe} (see also the references given there). Later, A. Gogatishvili and V. Kokilashvili \cite{GK}, working in a more general setting of ``homogeneous type general spaces'' (see the reference for precise definition) and with measures $\sigma$ and $\omega$ which are both absolutely continuous with respect to $\mu$, showed that \eqref{est:strong_for_M} is characterized by a $(B,B,B)$ testing condition with balls. 

We will provide a characterization of \eqref{est:strong_for_M} by a testing condition with dyadic cubes:

\begin{theorem}\label{thm:theoremA}
Suppose $0\leq \gamma <1$ and $1<p\leq q\leq \infty$, $p<\infty$. Let $(X,\rho,\mu)$ be a space of homogeneous type, and let $\sigma$ and $\omega$ be positive $\sigma$-finite Borel-measures on $X$. Let $M_{\mu ,\gamma}$ be the fractional maximal operator defined by \eqref{def:operatorM}. Then the strong type norm inequality
\begin{equation}\label{eq:strongtype_M}
\| M_{\mu,\gamma}f\|_{L^q_\omega}\leq N\| f\|_{L^p_\sigma}
\end{equation}
holds for all $f\in L^p_\sigma$, if and only if $\mu \ll \sigma$ and the testing condition
\begin{equation}\label{eq:testing;cond_M}
\left\| \chi_{Q}M_{\mu , \gamma}\left(\chi_{Q}\left[\frac{d\mu}{d\sigma}\right]^{1/(p-1)}\right)\right\|_{L^q_\omega}\leq N_1
\left\|\chi_{Q}\left[\frac{d\mu}{d\sigma}\right]^{1/(p-1)}\right\|_{L^p_\sigma}<\infty ,
\end{equation}
holds for all dyadic cubes $Q\in \bigcup_{t=1}^{L}\mathscr{D}^t$. Moreover, $N\leq cN_1$, where $c$ is a constant depending only on $X, \mu,\gamma , p$ and $q$. 
\end{theorem}
We begin by proving the necessity of the conditions $\mu \ll \sigma$ and \eqref{eq:testing;cond_M} for \eqref{eq:strongtype_M}. To prove the sufficiency, we will reduce to a dyadic analogue.

\begin{proof}[The proof that \eqref{eq:strongtype_M} implies both $\mu \ll \sigma$ and \eqref{eq:testing;cond_M}]
We may follow the proof of the dyadic analogue given in \cite[Theorem A]{S82} that (2.1) of \cite{S82} implies $\mu \ll \nu$ and (2.2). We repeat the details for the readers convenience.

To show $\mu \ll \sigma$, suppose for a contradiction that $E\subseteq X$ is a bounded Borel set with $\mu(E)>0=\sigma(E)$. Set $f=\chi_E$ in \eqref{eq:strongtype_M}. By Lemma~\ref{lem:positivemeasure}, there exists $Q\in\mathscr{D}^t$ with $\mu(Q\cap E)>0$ and $\omega(Q)>0$. By considering the containing ball $B(Q)$, we see that $M_{\mu ,\gamma} f >0$ on $B(Q)$, and it follows that the left hand side of \eqref{eq:strongtype_M} is positive while the right hand side is zero. This contradiction shows that $\mu \ll \sigma$. 

Let $u\in L^1_{\loc}(X,\sigma)$ be the Radon--Nikodym derivative of $\mu$ with respect to $\sigma$, i.e. $d\mu = u\, d\sigma$.
Suppose, for a contradiction, that for some dyadic cube $Q\in  \bigcup_{t=1}^{L}\mathscr{D}^t$ there holds
\[+\infty=\left\|\chi_{Q}u^{1/(p-1)}\right\|_{L^p_\sigma}=\| \chi_{Q}u\|_{L^{p'}_\sigma}^{p'/p}.\] 
By duality, it follows that there exists $f\in L^p_\sigma$ such that
\[\int_{X}f(\chi_{Q}u)\, d\sigma =\int_{Q}f\, d\mu=\infty . \]
This implies that $M_{\mu ,\gamma}f =  +\infty$ on $Q$ and consequently, $M_{\mu ,\gamma} f \equiv +\infty$. It follows that the left hand side of \eqref{eq:strongtype_M} is infinite while the right hand side is finite. This contradiction shows that 
\[\int_{Q} u^{p/(p-1)}\,d\sigma =\int_{Q} u^{p'}\,d\sigma < +\infty \]
for all dyadic cubes $Q\in  \bigcup_{t=1}^{L}\mathscr{D}^t$. Finally, by letting $f=\chi_{Q}u^{1/(p-1)}\in L^p_\sigma$
in \eqref{eq:strongtype_M}, we in particular obtain \eqref{eq:testing;cond_M}.
\end{proof}

\subsection{The dyadic maximal operator}
In order to proof the sufficiency of the testing condition in Theorem~\ref{thm:theoremA}, we will reduce to a dyadic analogue. Let $\mathscr{D}^t$ denote any fixed family of dyadic cubes $Q^k_\alpha$ (recall the definition and properties from Section~\ref{def;not;geom}). Suppose $\mu$ is a positive locally finite Borel-measure on $X$. For $0\leq \gamma <1$, we define \textit{the dyadic maximal operator} $M_{\mu ,\gamma}^{\mathscr{D}^t}$ by
\begin{equation}\label{def:dyadicoperatorM}
M_{\mu ,\gamma}^{\mathscr{D}^t}f(x):=\sup_{Q\in\mathscr{D}^t\atop \mu(Q)>0}\frac{\chi_{Q}(x)}{\mu(Q)^{1-\gamma}}\int_{Q}\abs{f}\, d\mu , \quad x\in X.
\end{equation}
We will usually assume that $\mu$ satisfies the doubling condition \eqref{def:doubling} but this assumption will then be indicated. Note that if $\mu$ is doubling and non-trivial then $0<\mu(Q)<\infty$ for all cubes $Q$.

The following pointwise inequalities concerning the operator $M_{\mu ,\gamma}$ and its dyadic counterparts were shown in \cite[Proposition 7.9]{oma} with $\gamma =0$. The proof for the case $0\leq \gamma <1$ is virtually identical.

\begin{lemma}\label{lem:eqv;M}
Suppose $\mu$ has the doubling property \eqref{def:doubling} and $f\in L^1_{\loc}(X,\mu)$. For every $x\in X$ we have the pointwise estimates
\begin{equation}\label{eq:eqv;M}
M_{\mu ,\gamma}^{\mathscr{D}^t}f(x)\leq CM_{\mu,\gamma}f(x) \quad\text{and} \quad
M_{\mu ,\gamma} f(x) \leq C\sum_{t=1}^{L}M_{\mu ,\gamma}^{\mathscr{D}^t}f(x).
\end{equation}
The constant $C\geq 1$ depends only on $X, \mu$ and $\gamma$, and the first inequality holds for every $t=1,\ldots ,L$.
\end{lemma}


Lemma~\ref{lem:eqv;M} shows that in order to prove the remaining part of Theorem~\ref{thm:theoremA}, we may reduce to the dyadic analogue. We will perform yet another reduction. 

\subsection{Dual weight trick}\label{subsec:dual;weight;trick} 
It is a standard part of the weighted theory to reformulate \eqref{eq:strongtype_M} by imposing same measure $v$ on both sides of \eqref{eq:strongtype_M}, as opposed to the two measures $\mu$ and $\sigma$. Thus, before turning to the dyadic analogue of Theorem~\ref{thm:theoremA}, it is convenient to recast \eqref{eq:strongtype_M} into a more ``natural'' form, which permits the replacement of the three measures $\mu$, $\sigma$ and $\omega$ by measures $v:=vd\mu$ (with $v$ to be chosen) and $\omega$, and which leads to an appearance of the testing conditions similar to the ones appearing in the other results of this paper. This reformulation also leads more naturally to the correct testing functions. To this end, we make a ``dual weight trick'' due to Sawyer. 

Assume that $\mu \ll \sigma$ and that the testing inequality \eqref{eq:testing;cond_M} holds for all dyadic cubes. Let $0\leq u\in L^1_{\loc}(X,\sigma)$ be the Radon--Nikodym derivative of $\mu$ with respect to $\sigma$,  i.e. $d\mu = u\, d\sigma$. We substitute $f=gv$ in \eqref{eq:strongtype_M}:
\[\| M_{\mu ,\gamma}(gv)\|_{L^q(X,\omega)}\leq C\| gv \|_{L^p(X,\sigma)}
=C\| g \|_{L^p(X,v^p\sigma)}.\]
Now choose $v$ such that $v^p=vu$, i.e. $v=\chi_{\{u>0\}}u^{1/(p-1)}\geq 0$. Note that the second inequality in \eqref{eq:testing;cond_M} ensures, in particular, that $v\in L^1_{\loc}(X,\mu)$. Indeed, if $E$ is a bounded set in $X$, there exists a dyadic cube $Q\in\bigcup_{t=1}^{L}\mathscr{D}^t$ such that $E\subseteq Q$. By the testing condition, 
\[\infty>\int_{Q}u^{p/(p-1)}\,d\sigma 
 =\int_{Q}\, u^{1/(p-1)}d\mu =\int_{Q}vd\mu \geq \int_{E}vd\mu.\]
The weight $v$ is then identified with the positive locally finite measure (denoted by the same symbol) $v(E):=\int_{E}vd\mu = \int_{E}v^pd\sigma$. Hence, an equivalent problem (the equation $f=gv$ defines a bijection $L^p_\sigma\to L^p_v$, $g\mapsto f$) for the sufficiency part of Theorem~\ref{thm:theoremA} is to show that \eqref{eq:testing;cond_M} implies
\begin{equation}\label{eq:strong_Mdy(2)}
\| M_\gamma(f\, dv) \|_{L^q_\omega}\leq N\| f\|_{L^p_v} \quad \text{for all $f\in L^p_v$,}
\end{equation} 
where $v=u^{1/(p-1)}$, $dv=vd\mu$, and $M_\gamma$ is an operator defined by
\begin{equation*}\label{def:M}
M_{\gamma}(f\, dv)(x) := \sup_{B}\frac{\chi_{B}(x)}{\mu(B)^{1-\gamma}}\int_{B}\abs{f}\,dv ,\quad x\in X .
\end{equation*}
Its dyadic counterpart is given by
\begin{equation}\label{def:M_reduct}
M^{\mathscr{D}^t}_{\gamma}(f\, dv)(x) := \sup_{Q\in\mathscr{D}^t}\frac{\chi_{Q}(x)}{\mu(Q)^{1-\gamma}}\int_{Q}\abs{f}\,dv ,\quad \quad x\in X .
\end{equation}
We have dropped the subscript $\mu$ on the notation emphasizing the fact that the integration inside the operator is now with respect to another measure. The advantage of the stated reformulations is that the same measure appears inside the operator $M_\gamma$ and in the norm on the right side of \eqref{eq:strong_Mdy(2)}. The testing condition \eqref{eq:testing;cond_M} of Theorem~\ref{thm:theoremA} may similarly be reformulated as
\begin{equation}\label{eq:testing;cond_M(2)}
\left(\int_{Q}M_\gamma(\chi_{Q}\, dv)^q\, d\omega\right)^{1/q}\leq N_1 v(Q)^{1/p},
\end{equation}
having the appearance similar to the testing conditions in Theorem~\ref{thm:theoremB}.

\bigskip

By the implemented dual weight trick, proving the remaining part of Theorem~\ref{thm:theoremA} is reduced to proving that \eqref{eq:testing;cond_M(2)} implies \eqref{eq:strong_Mdy(2)} for $v$ depending on $\mu$. In the following we will consider this estimate for general $\sigma$ which needs not be related to $\mu$.  The proof of Theorem~\ref{thm:theoremA} is now completed by the following Proposition.

\begin{proposition}\label{prop:theoremAdy}
Suppose $t=1,\ldots,L$ and  let $0\leq \gamma <1$ and $1<p\leq q\leq \infty$, $p<\infty$. Let $\mu, \sigma$ and $\omega$ be positive $\sigma$-finite Borel-measures on a quasi-metric space $(X,\rho)$, and let $M^{\mathscr{D}^t}_{\gamma}(\cdot \, d\sigma)$ be the dyadic operator defined in \eqref{def:M_reduct}. Then
\[\|M^{\mathscr{D}^t}_{\gamma}\|_{L^p_\sigma\to L^q_\omega}\approx [\sigma,\omega]_{S_{p,q}}:=\sup_{Q\in\mathscr{D}^t}\sigma(Q)^{-1/p}\|\chi_{Q}M^{\mathscr{D}^t}_{\gamma}(\chi_Q\,d\sigma) \|_{L^q_\omega}. \]
The constant of equivalence only depends on $p$ and $q$. 
\end{proposition}

In case $X=\R^n$, this sort of dyadic result was proved by Sawyer \cite[Theorem A]{S82} where, for example, the Marcinkiewicz interpolation theorem between weak $(1,q/p)$ and strong $(\infty ,\infty)$ is applied for a suitable operator. We will present a slightly different argument even though the original (Euclidean) proof could be adapted just as well. Some of our argument, however, follows the same line as Sawyer's original proof in which case this will be indicated. 

\begin{remark}
In Proposition~\ref{prop:theoremAdy}, we do not need to assume that $\mu$ (the underlying measure that appears inside the operator $M^{\mathscr{D}^t}_\gamma$) has the doubling property \eqref{def:doubling}; it suffices to assume that $\mu$ is locally finite. Then, when defining $M_{\gamma}^{\mathscr{D}^t}(f\, d\sigma)(x)$, the supremum in \eqref{def:M_reduct} is over all dyadic cubes $Q$ with $\mu(Q)>0$. 
However, the passage from Proposition~\ref{prop:theoremAdy} to Theorem~\ref{thm:theoremA} via Lemma~\ref{lem:eqv;M} depends on the doubling property of $\mu$.
\end{remark}


\begin{proof}[Proof of Proposition \ref{prop:theoremAdy}]
Since the estimate $\gtrsim$ is clear, only the estimate $\lesssim$ requires a proof. Moreover, we may assume that $\mathscr{D}^t$ consists of dyadic cubes $Q^k_\alpha$ restricted to $k\geq k_0$ (i.e. the size of cubes is bounded from above); the proof will provide an estimate independent of $k_0$, and the Monotone Convergence Theorem will then complete the proof. 

Suppose $f\in L^p_\sigma$ and assume, without loss of generality, that $f$ is bounded with bounded support. First we make the elementary observation that for such $f$ (and for $\mathscr{D}^t$ with cubes that have size bounded from above), we have $M^{\mathscr{D}^t}_{\gamma}(f\,d\sigma)\in L^q_\omega$. To check this, suppose $f=\chi_B$ for a ball $B$. By similar considerations performed in the proof of Lemma~\ref{lem:finite}, we see that it suffices to show that
\[\|\chi_{R_1^c}M^{\mathscr{D}^t}_\gamma(\chi_{R_0}\,d\sigma)\|_{L^q_\omega}^q<\infty \]
for all top-level cubes $R_1\in\mathscr{D}^t$ and all dyadic children $R_0$ of $R_1$ that intersect $B$. To this end, fix such $R_0$ which intersects $B$ and note that
\[\chi_{R_1^c}M^{\mathscr{D}^t}_\gamma(\chi_{R_0}\,d\sigma) \leq 
\chi_{R_1^c}\sup_{Q\in\mathscr{D}^t}\frac{\chi_Q}{\mu(Q)^{1-\gamma}}\sigma(R_1\cap Q) = \sup_{Q\in\mathscr{D}^t\atop Q\subseteq R_1}\chi_Q\chi_{R_1^c}\frac{\sigma(Q)}{\mu(Q)^{1-\gamma}} =0\]
since $R_1$ is one of the top-level cubes (thus, no larger cube contains $R_1$), and thus $\chi_Q\chi_{R_1^c}(x)=0$ for all $Q\subseteq R_1$ and $x\in X$.

\bigskip

For every $k\in \Z$, consider
\[\Omega_k := \{x\in X\colon M^{\mathscr{D}^t}_\gamma(f\,d\sigma)(x)>2^k\} =\bigcup_{Q\in\mathscr{Q}_k}Q\]
where $\mathscr{Q}_k$ is the collection of dyadic cubes in $\mathscr{D}^t$ maximal, hence disjoint, relative to the collection of cubes with the properties that 
\[\mu(Q)>0\quad\text{and}\quad
\frac{1}{\mu(Q)^{1-\gamma}}\int_{Q}\abs{f}\, d\sigma>2^k .\]
Note that, by the choice of the cubes $Q\in \mathscr{Q}_k$, we have $\sigma(Q)>0$ and
\begin{align}
\mu(Q)^{1-\gamma } & <2^{-k} \int_{Q}\abs{f}\,d\sigma \label{def:collection;cubes}
\\
& \leq 2^{-k} \sigma(Q)^{1/p'}\left(\int_{Q}\abs{f}^p\,d\sigma\right)^{1/p} \label{M:proof_est:mu},
\end{align}
and that 
\begin{equation}\label{M:proof_est:M_for_testf}
\begin{split}
M^{\mathscr{D}^t}_\gamma(\chi_{Q}\,d\sigma)\geq \mu(Q)^{\gamma -1}\sigma (Q)\quad\text{on $Q$}.
\end{split}
\end{equation}

\textbf{Case 1:} $1<p<q=\infty$. This case is treated following the proof given in \cite{S82}. Let $Q\in \mathscr{Q}_k$ and suppose $\omega (Q)>0$. By \eqref{M:proof_est:M_for_testf} and the testing condition with $q=\infty$, we have
\begin{equation*}\label{eq:positive_on_Qkj}
\begin{split}
\mu(Q)^{\gamma -1}\sigma(Q)& \leq  M^{\mathscr{D}^t}_\gamma(\chi_{Q}\,d\sigma)(x)\leq 
\| \chi_{Q}M^{\mathscr{D}^t}_\gamma(\chi_{Q}\,d\sigma)\|_{L^\infty_\omega} \leq 
[\sigma,\omega]_{S_{p,q}}\sigma(Q)^{1/p} \quad \text{for $\omega$-a.e. $x\in Q$}. 
\end{split}
\end{equation*}
Since $\sigma(Q)$ is positive and finite, we obtain that $\mu(Q)^{\gamma-1}\sigma(Q)^{1/p'}\leq [\sigma,\omega]_{S_{p,q}}$. By this and \eqref{M:proof_est:mu} we conclude with
\[2^k\leq [\sigma,\omega]_{S_{p,q}}\| f\|_{L^p_\sigma}, \]
which shows that the set of integers $k$ for which $\omega (\Omega_k)>0$ is upper bounded. This completes the proof for the case $1<p<q=\infty$. 

\bigskip 

\textbf{Case 2:} $1<p\leq q<\infty$. 
For $Q\in\mathscr{Q}_k$ define $U_k(Q):= Q\setminus \Omega_{k+1}$. Note that the sets $U_k(Q)\subseteq Q$ are pairwise disjoint in both $Q$ and $k$, and that 
\[\Omega_k\setminus \Omega_{k+1}= \bigcup_{Q\in\mathscr{Q}_k}Q \setminus \Omega_{k+1} =\bigcup_{Q\in\mathscr{Q}_k} U_k(Q) .\]
We then estimate (recall that $M^{\mathscr{D}^t}_\gamma(f\, d\sigma)\in L^q_\omega$, and consequently $M^{\mathscr{D}^t}_\gamma(f\, d\sigma)<\infty$ a.e.)
\begin{align*}
\int_{X} (M^{\mathscr{D}^t}_\gamma(f\, d\sigma))^q\,d\omega &= 
\sum_{k}\underset{\{ 2^k<M^{\mathscr{D}^t}_\gamma(f\,d\sigma)\leq 
2^{k+1}\}}\int \!\!\!\!\!\!\!\!\!\!\!\!\!\!\!\!\!(M^{\mathscr{D}^t}_\gamma(f\,d\sigma))^q\,d\omega \\
&\leq \sum_{k}2^{(k+1)q}\omega(\Omega_k\setminus \Omega_{k+1})
=2^q\sum_{k}\sum_{Q\in\mathscr{Q}_k}2^{kq}\omega(U_k(Q)) \nonumber\\
& \leq 2^q \sum_{k}\sum_{Q\in\mathscr{Q}_k}
\omega(U_k(Q))\left(\mu(Q)^{\gamma -1}\int_{Q}\abs{f}\,d\sigma\right)^{q}\quad \text{by \eqref{def:collection;cubes}.}\nonumber
\end{align*}
Since $\sigma(Q)>0$ for $Q\in \mathscr{Q}_k$, we may divide by it and obtain
\begin{align}\label{eq:normestimateforMdy}
\int_{X} (M^{\mathscr{D}^t}_\gamma(f\, d\sigma))^q\,d\omega & 
\leq 2^q \sum_{k}\sum_{Q\in\mathscr{Q}_k}
\omega(U_k(Q))\left(\mu(Q)^{\gamma -1} \sigma(Q) \right)^q \left(\frac{1}{\sigma(Q)}\int_{Q}\abs{f}\, d\sigma \right)^q \nonumber\\
& \leq 2^q \sum_{k}\sum_{Q\in\mathscr{Q}_k}
\int_{U_k(Q)}(M^{\mathscr{D}^t}_\gamma(\chi_{Q}\,d\sigma))^q\,d\omega\left(\langle f\rangle^\sigma_Q \right)^q\quad\text{by \eqref{M:proof_est:M_for_testf}.}
\end{align}

Recall the so-called principal cubes from \ref{principalcubes} with the properties listed in Remark~\ref{rem:principal}, and also recall the notation $\Pi(Q)\in \mathscr{P}$ for the smallest principal cube containing $Q\in\mathscr{Q}_k$. We re-organize the summation in \eqref{eq:normestimateforMdy} as
\[\sum_{k}\sum_{Q\in\mathscr{Q}_k}=\sum_{P\in \mathscr{P}} \sum_{k}\sum_{Q\in\mathscr{Q}_k\atop \Pi(Q)=P} .\]

Note that by property (ii) of Remark~\ref{rem:principal}, and since $U_k(Q)\subseteq Q\subseteq P$ are disjoint in both $Q$ and $k$,
\begin{equation*}
\begin{split}
\sum_{k}\sum_{Q\in\mathscr{Q}_k\atop \Pi(Q)=P} \int_{U_k(Q)}(M^{\mathscr{D}^t}_\gamma(\chi_{Q}\,d\sigma))^q\,d\omega\left(\langle f\rangle^\sigma _{Q}\right)^q &\leq 
\left(2\langle f\rangle^\sigma _{P} \right)^q\sum_{k}\sum_{Q\in\mathscr{Q}_k \atop \Pi(Q)=P} \int_{U_k(Q)} (M^{\mathscr{D}^t}_\gamma(\chi_{P}\,d\sigma))^q\, d\omega \\
& \leq 2^q \left(\langle f\rangle^\sigma _{P} \right)^q \int_{P} (M^{\mathscr{D}^t}_\gamma(\chi_{P}\,d\sigma))^q\, d\omega \\
&\leq 2^q \left(\langle f\rangle^\sigma _{P} \right)^q \sigma(P)^{q/p}[\sigma,\omega]^q_{S_{p,q}}.
\end{split}
\end{equation*}
Hence, from \eqref{eq:normestimateforMdy} we deduce
\begin{equation*}
\begin{split}
\int_{X} (M^{\mathscr{D}^t}_\gamma(f\, d\sigma))^q\,d\omega & \lesssim [\sigma,\omega]^q_{S_{p,q}}\sum_{P\in \mathscr{P}} \left(\langle f\rangle^\sigma _{P} \right)^q \sigma(P)^{q/p} = [\sigma,\omega]^q_{S_{p,q}} \sum_{P\in \mathscr{P}}\Big(\sigma(P)\left(\langle f\rangle^\sigma _{P} \right)^p\Big)^{q/p}\\
&\leq  [\sigma,\omega]^q_{S_{p,q}}\left( \sum_{P\in \mathscr{P}}\sigma(P)\left(\langle f\rangle^\sigma _{P} \right)^p\right)^{q/p}\quad\text{since $q\geq p$}\\
&= [\sigma,\omega]^q_{S_{p,q}} \left[\int_{X}\sum_{P\in \mathscr{P}}\chi_{P}(x)\left(\langle f\rangle^\sigma _{P} \right)^p d\sigma(x)\right]^{q/p}\\
& \leq [\sigma,\omega]^q_{S_{p,q}} \left[\int_{X}(M_\sigma^{\mathscr{D}^t}f)^p d\sigma\right]^{q/p} \lesssim [\sigma,\omega]^q_{S_{p,q}}\| f\|_{L^p_\sigma}^q
\end{split}
\end{equation*}
where we used Lemma~\ref{lem:mainlemma} in second-to-last estimate, and the universal maximal function estimate \eqref{eq:universalmaximal} in the last estimate. 
\end{proof}

\def\cprime{$'$} \def\cprime{$'$} \def\cprime{$'$}

\end{document}